\documentclass[11pt,reqno]{amsart}

\usepackage{amsmath}
\usepackage{amssymb}
\usepackage{euscript}

\usepackage{geometry}
\usepackage{layout}
\usepackage{hyperref}

\newtheorem{theorem}{Theorem}
\newtheorem{lemma}[theorem]{Lemma}
\newtheorem{corollary}[theorem]{Corollary}
\newtheorem{proposition}[theorem]{Proposition}

\newtheorem{claim}[theorem]{Claim}
\newtheorem{definition}[theorem]{Definition}

\newcommand{\F}{{\mathbb F}}

\def\scalprod#1#2{({#1}\, |\, {#2})}
\def\dim{{\rm dim}}
\newcommand{\beeq}{\begin{eqnarray*}}
\newcommand{\eneq}{\end{eqnarray*}}

\def\Z{\mathbb Z}

\newcommand{\QQ}{\mathcal{Q}}
\newcommand{\LL}{\mathcal{L}}
\newcommand{\CC}{\mathcal{C}}
\newcommand{\RR}{\mathcal{R}}

\newcommand{\A}{\mathcal{A}}

\newcommand{\C}{{\mathbb{C}}}
\newcommand{\Q}{{\mathbb{Q}}}
\newcommand{\cS}{\mathcal{S}}
\newcommand{\SSS}{\mathcal{S}}
\newcommand{\TT}{\mathcal{T}}
\newcommand{\OO}{\mathcal{O}}
\newcommand{\cX}{\EuScript{X}}

\newcommand{\rk}{\operatorname{rk}}
\newcommand{\wt}{\operatorname{wt}}
\newcommand{\Gl}{\operatorname{Gl}}

\newcommand{\Rad}{\operatorname{Rad}}

\newcommand{\1}{{\bf \operatorname{1}}}
\newcommand{\codim}{\operatorname{codim}}
\newcommand{\dimn}{\dim_{}}

\newcommand{\be}{\begin{equation}}
\newcommand{\ee}{\end{equation}}

\author{Christine Bachoc} 
\address{Institut de Math\'ematiques de Bordeaux, UMR 5251, universit\'e de Bordeaux, 351 cours de la Lib\'eration, 33400 Talence, France.}
\email{Christine.Bachoc@math.u-bordeaux.fr}
\author{Oriol Serra}
\address{Universitat Polit\`ecnica de Catalunya, Matem\`atica Aplicada IV, M\`odul C3, Campus Nord,  Jordi Girona, 1, 08034 Barcelona, Spain.}
\email{oserra@ma4.upc.edu}
\author{Gilles Z\'emor}
\address{Institut de Math\'ematiques de Bordeaux, UMR 5251, universit\'e de Bordeaux, 351 cours de la Lib\'eration, 33400 Talence, France.}
\email{zemor@math.u-bordeaux.fr}

\title[]{An analogue of Vosper's Theorem for Extension Fields}

\begin{document}

\begin{abstract} We are interested in characterising pairs $S,T$ of
  $F$-linear subspaces in a field extension $L/F$ such that the linear
  span $ST$ of the set of products of elements of $S$ and of elements
  of $T$ has small dimension. Our central result is
  a linear analogue of Vosper's Theorem, which
  gives the structure of vector spaces $S, T$ in a prime extension $L$
  of a finite field $F$ for which
$$
\dim_FST =\dim_F S+\dim_F T-1,
$$
when $\dim_F S, \dim_F T\ge 2$ and $\dim_F ST\le [L:F]-2$.
\end{abstract}

\maketitle

\section{Introduction}
Inverse problems in additive theory aim to provide structural results
of sets in an additive group which have a small subset sum. Motivated
by  a problem on difference sets, Hou, Leung and Xiang \cite{hou2002}
obtained a linear analogue of one of the central results in the area, the
theorem of Kneser \cite{knesrcomp}, in which cardinalities of sets in
an abelian group are substituted by dimensions of subspaces over a
field. 
To be specific, let $F$ be a field and let $L$ be an extension field of
$F$. If $S$ and $T$ are $F$-subvector spaces of $L$, we shall denote
by $ST$ the $F$-linear span of the set of products $st$, $s\in S$,
$t\in T$. The linear analogue of Kneser's theorem can be stated as follows.

\begin{theorem}[Hou, Leung and Xiang] \label{thm:kneser}
Let $F\subset L$ be fields and let $S,T$ be $F$-subvector spaces  of
$L$ of finite dimension. 
Suppose that every algebraic element in $L$ is separable over $F$. Then
$$
\dimn ST\ge \dimn S+\dimn T-\dimn H(ST),
$$
where $H(ST)=\{x\in L: xST\subset ST\}$ denotes the stabilizer of $ST$ in $L$.
\end{theorem}

It was recently shown by the present authors that
Theorem~\ref{thm:kneser} continues to hold without the requirement
that every algebraic element in $L$ be separable over $F$ \cite{BSZ}.

One of the remarkable features of Theorem~\ref{thm:kneser} is that the
original theorem of Kneser can be obtained as a corollary. It can
therefore be viewed not only as a transposition of an additive result, but as a generalisation.
This gives extra motivation for studying the translation of additive theory to
its linear counterpart, since we may obtain insight on the original
additive results and methods.

Eliahou and Lecouvey \cite{eliahou2009} obtained similar linear
analogues of classical additive theorems including theorems of Olson
\cite{olson} and Kemperman \cite{kemperman} in nonabelian
groups. Lecouvey \cite{lecouvey2011} pursued this direction by
obtaining, among other extensions,  linear versions of the
Pl\"unecke--Ruzsa \cite{ruzsa} inequalities.

Going back to the statement of Theorem~\ref{thm:kneser},
since the stabilizer $H(X)$ of a finite dimensional subspace is a subfield
of $L$ of finite dimension, if $L/F$ is an extension with no proper
finite intermediate extension, then either $H(X)=F$ or, if $\dim_FL$
is finite, $H(X)=L$. Hence, one obtains a linear analogue of the Cauchy--Davenport
inequality, which states that in an extension without a proper finite intermediate
extension (in particular in finite extensions of prime degree), the
dimension of the product of two subspaces $S$ and $T$ is either the
whole extension, or has dimension at least
$\dimn S+\dimn T-1$.
Strictly speaking, if we do not suppose the extension to be separable,
we need to call upon the strengthened version \cite{BSZ} of
Theorem~\ref{thm:kneser}. We note however, that
the linear analogue of the Cauchy-Davenport Theorem
was obtained in full generality, without any separability hypothesis, by
Eliahou and Lecouvey (see \cite{eliahou2009} Theorem 6.2). Let us
state it formally. When $L/F$ is an extension we shall say that $F$
is {\em algebraically closed in $L$} if every element of $L$ that is
algebraic over $F$ is in $F$.

\begin{theorem}\label{thm:Cauchy} Let $L/F$ be an extension such that
 $F$ is algebraically closed in $L$. For every pair of subspaces
  $S,T\subset L$ of finite dimension over $F$,
  \begin{equation}
    \label{eq:cauchy}
    \dimn ST\ge \min\{\dimn L,\dimn S +\dimn T-1\}.
  \end{equation}
\end{theorem}

One of the first inverse statements of additive theory is the theorem
of Vosper \cite{vosper} which states that, in a group of prime order, a pair of sets
attaining equality in the Cauchy--Davenport inequality are, except for
some degenerate cases, arithmetic progressions. In the present paper,
our main result is a transposition of  Vosper's theorem to the linear
setting which reads:

\begin{theorem}\label{thm:main} Let $F$ be a finite field and let
$L$ be an extension of prime degree $p$ of $F$. Let $S, T$ be
subspaces of $L$ such that $2\le \dimn S, \dimn T$ and $\dimn ST\le p-2$. If
\begin{equation}
  \label{eq:dim(ST)}
  \dimn ST=\dimn S+\dimn T-1,
\end{equation}
then $S$ and $T$ have bases of the form
$\{g,ga,\ldots ga^{\dim S-1}\}$ and $\{ g', g'a,\ldots ,g'a^{\dim T
  -1}\}$ for some $g,g',a\in L$.
\end{theorem}

We shall see that the conclusion of Theorem~\ref{thm:main}
continues to hold if $L$ is replaced by an infinite-dimensional extension of the
finite field $F$ containing no algebraic extension of~$F$ besides $F$.
This last result will be seen to result from Theorem~\ref{thm:main}
through a valuation argument, and can also be derived by the same
methodology that leads to Theorem~\ref{thm:main}. 

The question naturally arises whether the conclusion of
Theorem~\ref{thm:main} holds for {\em any} extension~$L/F$ that does
not contain proper algebraic extensions of $F$,
not just when $F$ is a finite field. 
The answer turns out to be negative, and we shall give a
counter-example for some infinite extension~$L/F$, and also for a
finite extension. 
Our general strategy does not require the base field to be finite,
and we shall obtain some partial results in the case of arbitrary
fields~$F$,
but the complete picture remains unclear. 

The proof of Theorem~\ref{thm:main} will import ideas from additive
combinatorics, and also use results on quadratic forms and what has
become known as the linear programming method in the theory of
error-correcting codes. Since the route to Theorem~\ref{thm:main} is fairly
long and devious we give an overview.

\subsection*{Methodology and overview of the proof of
  Theorem~\ref{thm:main}.}
With a view to characterising spaces $S,T$ that achieve equality in \eqref{eq:cauchy},
we shall 
consider subspaces $A$ of $L$ of {\em minimum dimension} (but at least
equal to $2$) such that 
\begin{equation}
  \label{eq:SA}
  \dimn SA=\dimn S +\dimn A-1.
\end{equation}
We shall first remark that if we can prove that such a subspace
has dimension exactly~$2$, and has a basis of the form $1,a$, then it
follows fairly straightforwardly that $S$ must have a basis in
geometric progression.
From then on we work towards showing that such a subspace $A$ must
have dimension $2$, at least in the case when the base field is finite.

With this objective, we shall transpose to the linear setting the so--called
isoperimetric method 
in additive combinatorics, a method largely due to the late Hamidoune
\cite{hamidoune}, which has  proved to be a particularly useful tool
in dealing with additive problems. 
We shall show that the subspaces
$A$ of \eqref{eq:SA}, that we call {\em atoms} by analogy with the
classical additive setting, must satisfy intersection properties.
In particular we shall show that an atom $A$ must satisfy the
following:
$$\text{If}\; x,y,z,t \in A,
\quad \text{and}\quad
xy=zt\quad \text{then}\quad \{Fx,Fy\} = \{
Fz,Ft\}$$
We call a subspace with this property a {\em Sidon space} by analogy
with Sidon sets of integers which are sets such that, for
any $x,y,z,t$ in the set,
$$x+y=z+t \quad \text{implies} \quad \{x,y\}=\{z,t\}.$$
Furthermore, we shall arrive at the conclusion that atoms must have a
product of small dimension, namely: 
\begin{equation}
  \label{eq:dim(A^2)}
  \dimn A^2 =2\dimn A - 1.
\end{equation}
It remains to prove that this last property is incompatible with the
Sidon property for any subspace $A$ of dimension $n$ greater than $2$. To
obtain this we shall need more tools: we shall consider the space of
quadratic forms over the base field $F$ in the variables
$x_1,\ldots , x_n$ together with the natural mapping $\Phi$ into $A^2$
deduced from the correspondence:
$$x_1\mapsto a_1,\ldots , x_n\mapsto a_n$$
where $a_1,\ldots a_n$ is a basis of $A$. 
We make the observation that any non-zero quadratic form that is of the form
$\ell_1\ell_1'+\ell_2\ell_2'$ where $\ell_1,\ell_1',\ell_2,\ell_2'$ are
$F$-linear expressions in $x_1,\ldots , x_n$ cannot map to zero by
$\Phi$: this is a reformulation of the Sidon property. This motivates
introducing the notion of {\em weight} of a quadratic form $Q$.

We shall
say that the zero quadratic form has weight $0$, that 
$Q$ has weight $1$ if it is the product of two non-zero linear forms,
and inductively
for $t>1$ that it has weight $t$ if it is the sum of~$t$ quadratic
forms of weight $1$ and is not of weight $<t$. This slightly
non-standard notion is related to the quadratic form's rank, but
behaves somewhat differently. The Sidon property of~$A$ means
therefore that the kernel of the mapping $\Phi$ is a subspace~$\CC$ of the
space of quadratic forms with minimum non-zero weight at least $3$.
On the other hand, property \eqref{eq:dim(A^2)} implies that $\CC$ must
be a space of large dimension. The rest of the proof of
Theorem~\ref{thm:main}
consists in showing that such a large space $\CC$ of quadratic forms of
weight at least~$3$ cannot exist over finite fields.

Specifically, we shall show that if $F$ is the finite field with $q$
elements and $\CC$ is a set of quadratic forms over $F^n$, any two of which
differ by a quadratic form of weight at least~$3$, then 
the cardinality $|\CC|$ of $\CC$ must satisfy:
\begin{equation}
  \label{eq:(n-1)(n-2)/2}
  |\CC| < q^{(n-1)(n-2)/2}.
\end{equation}
This is the minimal result that we need to finish the proof of
Theorem~\ref{thm:main}. To the best of our knowledge, packing problems
in the space of quadratic forms endowed with the aforementioned weight 
distance have not been considered before, making
\eqref{eq:(n-1)(n-2)/2} of possibly independent interest.
Loosely connected packing problems have been studied before for the
related rank distance, in the space of bilinear forms \cite{Del1}, in
the space of alternating bilinear forms \cite{Del2}, and more recently
in the space of symmetric bilinear forms \cite{Sch1,Sch2}.
The above papers have all applied what has become known as the
Delsarte linear programming method, a powerful approach for deriving
upper bounds on the size of codes (packings) initiated in \cite{Del} (see
also \cite{DL}) in the framework of association schemes.
The results obtained in \cite{Del1,Del2,Sch1,Sch2} cannot be made to
yield \eqref{eq:(n-1)(n-2)/2} directly however, and we have needed to work on the problem from
scratch, also applying the Delsarte methodology.

The paper is organized as follows. 
Sections~\ref{sec:critical} to \ref{sec:sidon} first develop a theory that
is valid over any base field.
Section~\ref{sec:critical} sets up
matters by reducing the existence of bases in geometric progression to the existence of
critical spaces of dimension $2$. 
Section~\ref{sec:con} develops an extension field
analogue of Hamidoune's isoperimetric method: in particular the
intersection theorem for atoms (Theorem~\ref{thm:intersection}) is
proved. Section~\ref{sec:2atom} derives further properties of atoms
and proves Theorem~\ref{thm:twin} which can be seen as a weak version of Theorem~\ref{thm:main}
in the sense that
it is valid only for extensions of degree $n$ such that $n-2$ is
prime. It places no restriction on the base field $F$ however, and
requires only $F$ to be algebraically closed in the extension field $L$.
Section~\ref{sec:sidon} introduces Sidon
spaces, makes the point that, under the hypothesis of
Theorem~\ref{thm:main}, possibly generalised to arbitrary base fields,
atoms must be Sidon spaces, and connects them
to packings in spaces of quadratic forms. Theorem~\ref{thm:main} is
reduced to a lower bound
(Theorem~\ref{thm:sidon}) 
on the dimension of the square of Sidon space. Section~\ref{sec:codes}
develops the study of packings of quadratic forms necessary to prove Theorem~\ref{thm:sidon}
and hence conclude a proof of 
Theorem~\ref{thm:main}. Section~\ref{sec:transcendental} is devoted to
extensions of Theorem~\ref{thm:main} to the case of
infinite-dimensional extensions. Finally Section~\ref{sec:comments}
concludes with some comments. 

\section{Deriving Theorem~\ref{thm:main} from the existence of
  critical spaces of dimension $2$}\label{sec:critical}

In this section we reduce the existence of bases in arithmetic
progression to the existence of critical spaces of dimension $2$.
In particular we reduce
Theorem~\ref{thm:main} to showing the
existence of a subspace $A$ of dimension~$2$ satisfying
\eqref{eq:SA}, but
statements will more generally hold for arbitrary base fields.

We first make the remark that if a space satisfying \eqref{eq:SA}
exists, then without loss of
generality we may suppose that it contains the field unit element $1$,
by multiplying $A$ by an appropriate element of $L$ if necessary.

\begin{lemma}\label{lem:prog}
Let $L$ be an extension of $F$ such that $F$ is algebraically closed
in $L$.
Suppose that $S$ is an $F$--subspace of dimension $s$ of $L$, and that
$A$ is a subspace of dimension $2$ generated by $\{1,a\}$ such that
$\dimn AS =\dimn S+1\le \dimn L-1$. Then there exists $g\in S$ such
that $\{g,ga,\ldots ,ga^{s-1}\}$ is a basis for $S$.
\end{lemma}

\begin{proof} 
We have $\dimn AS=\dim (S+aS)=2\dimn S-\dim(S\cap aS)$
  so
\begin{equation}\label{eq:k-1}
\dim(S\cap aS)=s-1.
\end{equation}
We next show the result by induction on $s\geq 2$.

If $\dimn S=2$, let $g'$ generate $S\cap aS$, i.e. $S\cap
aS=Fg'$. Then, $g'=ag$ with $g\in S$. Moreover $g\notin aS$ otherwise
$g=\lambda g'=\lambda ag$ for some $\lambda\in F^*$ which would mean $a\in F$, so $\{g,ag\}$ is a basis of
$S$.

In the general case $\dimn S=s$, let $S'=S\cap aS$. Then, from
\eqref{eq:k-1}, $\dimn S'=s-1$; moreover $S'$
also satisfies $\dimn AS'=\dimn S'+1$. Indeed, we have
\begin{equation*}
\dimn S' +1\leq \dimn AS'\leq \dimn aS=\dimn S' +1
\end{equation*}
where Theorem \ref{thm:Cauchy}  gives the first inequality while the second follows from
$AS'=S'+aS'\subset aS$. By induction, $S'$ has
a basis of the form $\{g',g'a,\dots, g'a^{s-2}\}$. Since $g'\in aS$,
for some $g\in S$, $g'=ga$. Moreover $g\notin S'$ otherwise 
$g=\sum_{i=1}^{s-1} \lambda_i ga^i$ for some $\lambda_i\in F$ but this
would mean that $\deg_F(a)\leq s-1<\dimn L$.
So, $\{g,ga,\dots,ga^{s-1}\}$ is a basis of $S$.
 \end{proof}

\begin{lemma}\label{lem:T}
Let $L/F$ be an extension with $F$ algebraically closed in $L$.
  Let $S, T$ be subspaces  of $L$ with $\dimn S, \dimn T\geq 2$. 
Suppose that there exists a basis of $S$ of the form $\{ g,
  ga, \ldots , ga^{s-1}\}$ for $a,g\in L$ and that we have
$$
\dimn ST =\dimn S +\dimn T -1\leq\dimn L -1.
$$
Then there is a basis of $T$ of the form $\{ g', g'a,\ldots
,g'a^{t-1}\}$
for some $g'\in L$.
\end{lemma}

\begin{proof} Without loss of generality we will assume that $S$ has a
  basis of the form $\{ 1,  a, \ldots , a^{s-1}\}$ and we will proceed
  by induction on $s=\dimn S\geq 2$. The case $s=2$ is treated in
  Lemma \ref{lem:prog}.

Let $S'$ be the subspace generated by $a,\ldots, a^{s-1}$: because $ST=T+S'T$, we have 
\begin{equation*}
\dimn T+\dimn S'T-\dim(T\cap S'T)=\dimn ST =\dimn S+\dimn T-1
\end{equation*}
thus
\begin{equation}\label{eq:prime}
\dimn S'T=\dimn S'+\dim(T\cap S'T).
\end{equation}
Since
$$\dimn S'+\dimn T-1\leq \dimn S'T\leq\dimn ST=\dimn S+\dimn T-1,$$
the leftmost inequality coming from Theorem~\ref{thm:Cauchy},
we have that \eqref{eq:prime}
leaves two possibilities: either $\dim(T\cap S'T)=\dimn T-1$ or
$\dim(T\cap S'T)=\dimn T$. Let us rule out the latter: indeed, it
would mean that $ST=S'T$, therefore, setting
$W=T+aT+\cdots +a^{s-2}T$ we get
$W\subset ST=S'T=aW$. But $W$ and $aW$ have the same dimension, 
so we would have $W=aW$, which is impossible, given that
$0<\dimn W<\dimn L$ and $F(a)$ is either infinite-dimensional or equal
to $L$.

So $\dimn S'T=\dimn S'+\dimn T-1$  and $\dimn S'=s-1$. Therefore, by induction 
$T$ has a basis of the required form.
\end{proof}

\section{Connectivity in Field Extensions}\label{sec:con}

We now transpose to the context of field extensions the
basic notions of the isoperimetric method as introduced in \cite{HALG}
and developed in a number of later papers, see \cite{YOH13}.
We borrow the terminology of \cite{HACTA,y2008} and other papers with some
adaptation to the linear case.

Recall that in any field extension $L/F$ any
non-zero $F$-linear form $\lambda$ defines the non-degenerate
symmetric bilinear form
  $$(x,y) \mapsto \scalprod{x}{y} = \lambda(xy).$$
Crucial to the
developments below will be the property:
\begin{equation}
  \label{eq:frobenius}
  \scalprod{xy}{z}=\scalprod{x}{yz} \; \mbox{ for all } x,y,z\in L.
\end{equation}

We make the remark that the statements of this section rely mostly on
\eqref{eq:frobenius} and as such could be stated in the general
context of algebras over $F$ in which a non-degenerate symmetric
bilinear form satisfying \eqref{eq:frobenius} exists, namely symmetric
Frobenius algebras. Since our applications will only concern extension
fields, we do not pursue this generalisation.

In the rest of this section, we consider an arbitrary
field extension $L/F$ together with a non-zero finite-dimensional
$F$-subspace $S$ of $L$. When $L$ is itself of finite dimension, we
shall consider it endowed with a fixed non-degenerate bilinear form
$\scalprod{}{}$ satisfying \eqref{eq:frobenius}.
In this case, for a subspace $X\subset L$ we denote by
$$X^\perp = \{y\in L:\; \forall x\in X, \scalprod{x}{y}=0\}.$$

For every subspace $X$ of $L$ with non-zero, finite dimension we denote by
$$
\partial_S X=\dimn XS-\dimn X,
$$
the increment of dimension of $X$ when multiplied by $S$. We have the submodularity relation:

\begin{proposition}\label{prop:subm} Let $X, Y$ be finite-dimensional subspaces of $L$. We have
$$
\partial_S (X+Y)+\partial_S (X\cap Y)\le \partial_S X+\partial_S Y.
$$
\end{proposition}

\begin{proof}
We have  $(X+Y)S\subset XS+YS$ and  $(X\cap Y)S\subset XS\cap YS$. Therefore,
\begin{eqnarray*}
\partial_S (X+Y)+\partial_S (X\cap Y)&=&\dimn (X+Y)S-\dim (X+Y)+\dimn (X\cap Y)S\\
&&-\dim (X\cap Y)\\
&\le &\dimn XS+\dimn YS-\dim (XS\cap YS)-\dim (X+Y)\\
&&+\dim (XS \cap YS)-\dim (X\cap Y)\\
&\le &\dimn XS+\dimn YS-\dimn X-\dimn Y\\
&=&\partial_S X+\partial_S Y .
\end{eqnarray*}
\end{proof}

Let $\cX_k$ be the set of subspaces $X$ of $L$ such that
$$k\le \dimn X <\infty \text{ and } \dimn XS+k\le \dimn L.$$
If the set $\cX_k$ is non-empty, we define the {\em $k$--th
  connectivity} of $S$ by
$$
\kappa_k (S)=\min_{X\in\cX_k} \partial_S X.
$$
If the set $\cX_k$ is empty we set $\kappa_k (S) =-\infty$.
When $\kappa_k (S)\neq -\infty$, we define
a {\em $k$--fragment of~$S$} to be a subspace $M$ of $\cX_k$ with $\partial_S M=\kappa_k
(S)$. A $k$--fragment with minimum dimension is called a {\em $k$--atom}. 

The following Lemma is crucial to the development
of the isoperimetric method.

\begin{lemma}\label{lem:F*}
 Let the extension $L/F$ be finite-dimensional. If $X$ is a
 $k$-fragment of $S$, then $X^*=(XS)^\perp$ is also a $k$-fragment of $S$.
\end{lemma}

\begin{proof}
 Since  $$0=\scalprod{xs}{x^*}   = \scalprod{x}{x^*s},$$
  for every $x\in X, s\in S$ and $x^*\in X^*$, we have
  $(X^*S)^\perp \supseteq X$, so that
  \begin{equation}
    \label{eq:n-k}\dimn X^*S \leq\dimn L - \dimn X.
  \end{equation}
It follows that
  \begin{align}
    \partial_S X^*& = \dimn X^*S -\dimn X^* \nonumber\\
    &\leq \dimn L - \dimn X
    -\dimn X^*\nonumber\\
    &=\dimn L - \dimn X - (\dimn L - \dimn XS)\nonumber\\
    &=\partial_S X.
 \label{eq:fragment}
  \end{align}
  Finally, since $X$ satisfies $\dimn XS\leq \dimn L-k$ we have
  that $\dimn X^* \geq k$. Together with \eqref{eq:n-k} and
  \eqref{eq:fragment} this implies that $X^*$ is a $k$-fragment of
  $S$.
\end{proof}

\begin{corollary}\label{cor:A<A*}
Let the extension $L/F$ be finite-dimensional. If $A$ is a $k$-atom of
  $S$ then
  \begin{equation}\label{eq:2a+k}
  \dimn L \geq 2\dimn A + \kappa_k(S).
  \end{equation}
\end{corollary}

\begin{proof}
  Let $A^*=(AS)^\perp$. Since
  $\dimn AS=\dimn A+\partial_SA=\dimn A+\kappa_k(S)$ by definition
  of $\partial_SA$ and $\kappa_k(S)$, we have that
  $\dimn L = \dimn AS + \dimn A^* = \dimn A+\kappa_k(S)
  +\dimn A^*$. Furthermore, by definition of an atom we have
  $\dimn A\leq\dimn A^*$ since $A^*$ is a fragment by
  Lemma~\ref{lem:F*}. Hence the result.
\end{proof}

The cornerstone of the isoperimetric method is the following property.

\begin{theorem}[Intersection Theorem]\label{thm:intersection} 
Let $A, B$ be two distinct
  $k$--atoms of $S$. Then,  $$\dim (A\cap B)\le k-1.$$
\end{theorem}

\begin{proof}  By the submodularity relation one has
$$
\partial_S (A+B)+\partial_S (A\cap B )\le \partial_S A+\partial_S B=2\kappa_k(S).
$$
Suppose that  $\dim (A\cap B)\ge k$. By the definition of a $k$--atom,  we have $\partial_S (A\cap B)>\kappa_k (S)$. It follows that $\partial_S (A+B)<\kappa_k (S)$. Hence,
\begin{eqnarray*}
\dim (A+B)S&=&\dim (A+B)+\partial_S (A+B)\\
&<&2\dimn A-\dim (A\cap B)+\kappa_k (S)\\
&\le &2\dimn A-k+\kappa_k (S),
\end{eqnarray*}
from which we get
$$\dim (A+B)S +k \leq \dimn L$$
trivially if $L$ is infinite-dimensional and by 
Corollary~\ref{cor:A<A*} if $L$ is finite-dimensional.
By the definition of $\kappa_k (S)$, this contradicts $\partial_S (A+B)<\kappa_k (S)$.
\end{proof}

Let $A$ be a $k$--atom of $S$. We observe that, from the definitions of
$\partial_S$ and of atoms, for each $\alpha\in
L\setminus\{0\}$, $\alpha A$ is also a $k$--atom of
$S$. Therefore there is a $k$--atom of $S$ containing $1$. By the
Intersection Theorem, when $k=1$ the atom containing $1$ is
unique. The following theorem is the linear analogue of a theorem of
Mann \cite[Ch. 1]{mann}. It is not as powerful as Kneser's theorem but
it is already enough to recover the linear Cauchy-Davenport Theorem (Theorem~\ref{thm:Cauchy}).

\begin{theorem}\label{thm:atom} Let $A$ be the $1$--atom of $S$
  containing $1$. 
Then  $A$  is a subfield of $L$. Moreover, if
$$
\dimn ST <\dimn S +\dimn T -1<\dim L,
$$
for some subspace $T$, then $A$ is a  subfield of $L$ properly containing $F$.
\end{theorem}

\begin{proof}  For each $a\in A\setminus\{0\}$ we have $a \in aA\cap
  A$. 
Hence, by the Intersection Theorem $aA=A$. Therefore $a^{-1}\in A$ and
$A$ is a subfield of $L$. Moreover, if $\dimn ST<\dimn S +\dimn T-1$ for some subspace $T$ then
$$
0\le \partial_S(A)=\dimn AS -\dimn A \le \partial_S T =\dimn ST -\dimn
T <\dimn S -1,
$$
so that we cannot have $\dimn A =1$.
\end{proof}

One consequence of Theorem \ref{thm:atom} is the linear
Cauchy--Davenport inequality
of Theorem~\ref{thm:Cauchy}
\begin{equation}\label{eq:cd}
\dimn ST \ge \dimn S +\dimn T -1,
\end{equation}
when $L$ has no proper finite-dimensional subfields containing~$F$.

 \section{The $2$--atom of a  Vosper space}\label{sec:2atom}

We now investigate the properties of 2-atoms of a subspace $S$, with
in particular the goal of showing
that under the conditions of Theorem~\ref{thm:main} they must be of
dimension 2. Note that under the conditions of Theorem~\ref{thm:main},
$2$-atoms of $S$ must exist.

Let $L/F$ be an extension such that $F$ is algebraically closed in $L$.
If $F$ is a finite field and
if $L$ is of finite dimension over $F$, this means that $L/K$ has prime
degree, whence the hypothesis of Theorem~\ref{thm:main}, but the
results of this section hold in the more general case.

Let $S$ be an $F$-vector space of finite dimension in $L$ such that
$$
\kappa_2 (S)=\dimn S-1.
$$
Let $A$ be a $2$--atom of $S$. We shall be interested in the sequence
of subspaces $A^i$, $i\geq 1$, where $A^i$ is defined inductively by $A^i=A^{i-1}A$.
In this section we shall show that $A$ is also a  $2$--atom of $A$, that
$A^i$, $i\ge 1$ is a  $2$--fragment of $A$ as long as
$\dimn A^iS +2\leq \dimn L$, and that in the case when $\dimn L$ is finite
$\dimn L\equiv 2 \pmod{n-1}$, where $n=\dimn A$.

\begin{lemma}\label{lem:2atom} If $A$ is a $2$--atom of $S$ then $A$ is a $2$--atom of $A$.
\end{lemma}

\begin{proof} The statement is clear if $\dimn A=2$. Suppose that
  $\dimn A =n\ge 3$. Note that $S$ is a witness that $\kappa_2
  (A)=\dimn A -1$. Let $B$ be a $2$--atom of $A$. Without loss of
  generality we assume $1\in A\cap B$.

Let $b\in B\setminus F$ and $a\in A\setminus F$.

\begin{claim} $\dimn A =\dimn B$ and $AB=A+bA=B+aB.$
\end{claim}
\begin{proof}
By the Intersection Theorem (Theorem~\ref{thm:intersection}),
$$
2\dimn A-1\le \dim (A+bA) \le \dimn AB=\dimn A+\dimn B-1,
$$
which implies $\dimn B\ge \dimn A$. Analogously,
$$
2\dimn B-1\le \dim (B+aB) \le \dimn AB=\dimn A +\dimn B-1,
$$
so that $\dimn A =\dimn B=n$.  Moreover
$$
AB=A+bA=B+aB.
$$
\end{proof}

\begin{claim}\label{claim:<dimL} $\dimn A^2B< \dimn L$.
\end{claim}
\begin{proof}
We have
$$
A^2B=A(AB)=A(aB+B)\le aAB+AB.
$$
Hence,
\begin{align}
\dimn A^2B &\le 2\dimn AB-\dim (aAB\cap AB)\nonumber\\
&\le 2(\dimn A+\dimn B-1)-\dimn B\nonumber\\
&=2\dimn A+\dimn B-2,\label{eq:a2b}
\end{align}
where in the second inequality we use that $aB\subset aAB\cap
AB$. On the other hand,
by Corollary~\ref{cor:A<A*},
$$
\dimn L\ge 2\dimn A+\kappa_2 (A)=3n-1.
$$
\end{proof}
Claim~\ref{claim:<dimL} allows us to apply the
linear Cauchy--Davenport inequality \eqref{eq:cd} to
the spaces $A^2$ and $B$, and together with \eqref{eq:a2b} we obtain
$$
\dimn A^2+\dimn B-1\le \dimn A^2B\le 2\dimn A+\dimn B-2,
$$
which implies
$$
\dimn A^2\le 2\dimn A-1.
$$
This shows that $A$ is its own $2$--atom.
\end{proof}

\begin{lemma}\label{lem:ak} Let $A$ be a $2$--atom of $S$ and $t\ge 2$ an integer. We have
$$
\dimn A^t=\min\{ \dimn A^{t-1}+\dimn A-1,\dimn L\}.
$$
\end{lemma}

\begin{proof} Let $a\in A\setminus F$. By Lemma  \ref{lem:2atom} $A$
  is a $2$--atom of $A$ and $\dimn A^2=\dim (A+aA)=2\dimn A-1$,
  which establishes the result for $t=2$. In particular, $A^2=A+aA$,
  so that we have
$$
A^t=A^{t-2}(A+ aA)=A^{t-1}+ aA^{t-1}.
$$
Now, notice that $A^{t-1}\cap aA^{t-1}$ contains $aA^{t-2}$, therefore
$$
\dimn A^{t}=\dim (A^{t-1}+aA^{t-1})\leq
\dimn A^{t-1}+\dimn A^{t-1} -\dimn A^{t-2}.
$$
Suppose that $A^t\neq L$.
By induction on $t$ we have that $\dimn A^{t-1} -\dimn A^{t-2}
=\dimn A-1$
therefore
$$
\dimn A^{t}\leq \dimn A^{t-1}+\dimn A-1.
$$
Now since  $A^t\neq L$, the linear Cauchy-Davenport inequality \eqref{eq:cd}
gives
$$
\dimn A^{t-1}+\dimn A-1\leq\dimn A^{t},
$$
which yields $\dimn A^{t}=\dimn A^{t-1}+\dimn A-1$.
\end{proof}

\begin{lemma}\label{lem:dim(A)even} Let $A$ be a $2$--atom of $S$ with
  $\dimn A=n>2$. Then
if $\dimn L$ is finite we have:
$$
\dimn L\equiv 2 \bmod{(n-1)}.
$$
\end{lemma}

\begin{proof}
Let $t$ be the largest positive integer such that $\dimn A^t<\dimn L$. 
We recall that $\dimn A^2\le \dimn L-2$, so that $t\ge
2$. Let $X=A^{t-1}$ and
let $X^*=(A^{t})^{\perp}$. Since we have supposed $\dimn A^t<\dimn L$,
we have $\dimn X^*>0$. 
Lemma~\ref{lem:ak} implies that either $\dimn X^*=1$ or $X=A^{t-1}$ is a $2$-fragment of
$A$. But if $X$ is a $2$-fragment of
$A$ then Lemma~\ref{lem:F*} applied to $X$ and
$A$ implies that $X^*$ is also a $2$-fragment of $A$.
But then the maximality of $t$ such that $\dimn A^t<\dimn L$ and
Lemma~\ref{lem:ak}
imply that $\dimn X^*\leq \dimn A-1$, which is in contradiction with
$A$ being a $2$-atom of $A$ (Lemma~\ref{lem:2atom}). Therefore
$\dimn A^t=\dimn L-1$.
To conclude, observe that Lemma~\ref{lem:ak} gives $\dimn A^t =
\dimn A+(t-1)(\dimn A-1)$,
whence
$$\dimn L= 2 + t(\dimn A-1).$$
\end{proof}

As an immediate consequence of Lemma~\ref{lem:dim(A)even}, we have
that if $\dimn L-2=m-2$ is prime, then atoms can only have dimension $2$ and by
Lemmas~\ref{lem:prog} and~\ref{lem:T} the linear version of Vosper's theorem holds for any 
extension of degree $m$ over any field, given that there is no
intermediate subfield. This is the case in particular if $p$ and $p-2$ are a pair of twin primes and if
$L/F$ is of degree $p$. Specifically:

\begin{theorem}\label{thm:twin}
  Let $L/F$ be a finite extension such that $F$ is algebraically
  closed in $L$. Suppose
furthermore that $[L:F]-2$ is prime.
Let $S, T$ be subspaces of $L$ such that $2\le \dimn S, \dimn T$ and
$\dimn ST\le \dimn L-2$. If
$$
\dimn ST=\dimn S+\dimn T-1,
$$
then there are bases of $S$ and of $T$ respectively of the form 
$$\{g,ga,\ldots
ga^{\dim S-1}\}\quad \text{and} \quad 
\{ g', g'a,\ldots ,g'a^{\dim T-1}\}$$
for some $g,g',a\in L$.
\end{theorem}

Without the condition that $\dimn L-2$ is prime some possible dimensions
other than $n=2$ escape Lemma~\ref{lem:dim(A)even}, although it certainly rules out the case
where $n$ is odd and $\dimn L$ is a prime, a fact which will be used
later on. To remove the remaining cases we introduce more tools.

\section{Sidon spaces and quadratic forms}\label{sec:sidon}
The Intersection Theorem~\ref{thm:intersection} gives a key property
of $2$--atoms that in extensions $L/F$ without proper
finite sub-extensions of $F$, translates into:
\begin{equation}\label{eq:sidon}
\forall x\in L\setminus F, \quad \dim (A \cap xA)\le 1.
\end{equation}

We note that  \eqref{eq:sidon} implies
$$
\forall x,y,z,t\in A\setminus\{0\},\; xy=zt  \Rightarrow
\{ Fx,Fy\}=\{Fz,Ft\}
$$
This is 
because since $x^{-1}A\cap z^{-1}A$ contains $F$, we deduce
that either $x$ and $z$ are $F$-proportional, from which the conclusion
follows, or that $x^{-1}A\cap z^{-1}A=F$ by \eqref{eq:sidon} from
which it follows that $x^{-1}t$ and $z^{-1}y$ each generate the
constant field $F$.

This suggests calling a subspace which satisfies \eqref{eq:sidon} a
{\em Sidon space}, by analogy with classical Sidon sets. We make the
remark that when a space $A$ has a basis in geometric progression then
there exists an $x\in L$ such that $\dim(A\cap
xA)=\dimn A-1$. Therefore Sidon spaces can be thought of as spaces
that are ``furthest away'' from a space with a basis in geometric progression.

For a Sidon space we clearly have
\begin{equation}
  \label{eq:sidon2}
  \dimn A^2\ge \dim (A+aA)\ge 2\dimn A-1,
\end{equation}
for any $a\in A\setminus F$. 
According to Lemma~\ref{lem:2atom}, 
if $A$ is a
$2$--atom of some set $S$ satisfying \eqref{eq:SA} in an extension
without proper finite subextensions, then the inequalities in \eqref{eq:sidon2} are actually equalities.
To recap we have:

\begin{proposition}\label{prop:sidon}
Let $L/F$ be an extension 
and suppose $F$ algebraically closed in $L$. Let $S$ be a subspace of
$L$ with finite dimension $\geq 2$ and let $A$ be a $2$-atom of $S$.
Then $A$ is a Sidon space. Furthermore if
$\dimn SA=\dimn S+\dimn A-1$ then $\dimn A^2=2\dimn A-1$.
\end{proposition}

Now classical Sidon sets $\cS$ of integers (or Sidon sets in abelian groups)
have the property that $|\cS + \cS| = \binom{|\cS|+1}{2}$ which
implies in particular that $|\cS + \cS| = 2|\cS|-1$ if and only if
$|\cS| \leq 2$. If we let ourselves be guided by the additive analogy, we
may be led to expect for a moment that for any Sidon space $A$ it is true that 
$\dimn A =\binom{\dim A+1}{2}$. This would immediately imply that the only Sidon
spaces $A$ such that $\dimn A^2 = 2\dimn A-1$ are of dimension
$\leq 2$. It is however {\em not true} in general that $\dimn A
=\binom{\dim A+1}{2}$ for Sidon spaces, even in extensions
of prime order.
 
For the sake of explicitness, here is a counter-example:
take for $F$ the finite
field of size~$2$, for $L$ the field of size $2^{19}$, and let $A$ be the
$F$-vector with basis $(1,\alpha,\alpha^7,\alpha^{12}+\alpha^2+1)$
where $\alpha$ is a root of the irreducible polynomial 
$X^{19} + X^{14} + X^{10} + X^7 + X^2 + X + 1$. Computing dimensions yields
$\dimn A^2=9<\binom{\dim A+1}{2}=10$.

However, for the purpose of deriving Theorem~\ref{thm:main}, we only
need to prove that the only Sidon spaces that
satisfy \eqref{eq:sidon2} with equalities are of dimension $\leq 2$. 
If $F$ is allowed to be any field, again this is not true.
After we make the connection between Sidon spaces and quadratic forms,
we shall, at the end of this section, give an example of a Sidon space of dimension $3$ that
satisfies \eqref{eq:sidon2} with equalities.

When $F$ is a finite field it {\em is} true though that Sidon spaces
satisfying \eqref{eq:sidon2} with equalities must be of dimension
$\leq 2$. Specifically:
\begin{theorem}\label{thm:sidon}
 Let $F$ be a finite field and let $L$ be an extension field of
 $F$. Let $A$ be a Sidon subspace of $L$ of finite dimension
 $\dimn A\geq 3$. Then $\dimn A^2>2\dimn A-1$.
\end{theorem}
Note that we do not need to suppose that $L/F$ is finite and of prime
degree. 

Theorem~\ref{thm:sidon} together with Proposition~\ref{prop:sidon}
will in particular establish that the $2$-atom
of $S$ in Theorem~\ref{thm:main} can only have dimension $2$, and
Theorem~\ref{thm:main} will follow from Lemmas~\ref{lem:prog} and~\ref{lem:T}.

To prove Theorem~\ref{thm:sidon} we shall transform our problem into a
problem in the space of quadratic forms.

Let $\QQ_n$ denote the vector  space  of homogeneous polynomials of
degree $2$ 
in the variables $x_1,\dots, x_n$ with coefficients in the field
$F$. A typical element of $\QQ_n$ will be denoted 
\begin{equation}
Q=\sum_{1\leq i\leq j\leq n} a_{ij} x_i x_j.
\end{equation} 
The $F$-vector space $\QQ_n$ is of dimension $n(n+1)/2$ and we can
identify it with the space of quadratic forms over $F^n$. Similarly,
let $\LL_n$ denote the space of linear forms over $F^n$, identified
with the space of homogeneous polynomials of degree $1$ in $x_1,\dots,
x_n$.
Let us now introduce a notion of {\em weight} of a quadratic form.
\begin{definition}
  For a non-zero quadratic form $Q\in\QQ_n$, let its {\em weight} equal the
  smallest integer $k$ such that $Q$ can be expressed as a sum of $k$
  products of linear forms in $x=(x_1,\ldots ,x_n)$.
\begin{equation}
\wt(Q):=\min\{k \ :\ Q=\ell_1(x)\ell_1'(x)+\dots+\ell_k(x)\ell_k'(x),
\ \ell_i,\ell_j'\in \LL_n, 1\leq i,j\leq k\}.
\end{equation}
\end{definition}

If $\CC$ is a set of quadratic forms, we will
call the minimum weight of $\CC$ the smallest weight of the difference
between two distinct quadratic forms of $\CC$:
\begin{equation}
\wt(\CC):=\min\{ \wt(Q-Q')\ : \ (Q,Q')\in \CC\times \CC, \ Q\neq Q'\}.
\end{equation}
We note that, if $\CC$ is a linear subspace of $\QQ_n$, then 
$\wt(\CC)=\min\{ \wt(Q)\ : \ Q\in \CC, \ Q\neq 0\}$.

Now let $A$ be a Sidon space of dimension $n$ in some extension of $F$. Let $(a_1,\ldots
,a_n)$ be a basis of $A$. Consider the
  homomorphism $\Phi$ of vector spaces $\LL_n\rightarrow A$ defined by the mapping
$$x_1\mapsto a_1, \ldots ,x_n\mapsto a_n.$$
This homomorphism extends in a natural way to a homomorphism of
$F$-vector spaces
$$\Phi~: \QQ_n \rightarrow A^2$$
through the relations $x_ix_j\mapsto a_ia_j$. Note that for any
$\ell,\ell'\in\LL_n$, the map $\Phi$ satisfies
$\Phi(\ell\ell')=\Phi(\ell)\Phi(\ell')$. Consider the subspace
$\CC$ of $\QQ_n$, $\CC=\ker\Phi$.

\begin{proposition}
 $A$ is a Sidon space if and only if
 for any $Q\in\CC, Q\neq 0$, we have $\wt(Q)\geq 3$.
\end{proposition}
\begin{proof}
  That $\wt(Q)>1$ is simply that elements of $A^2$ live in a
  field where products of non-zero elements are non-zero. If
  $\wt(Q)=2$, meaning that $Q=\ell_1\ell_1' + \ell_2\ell_2'$, then
  setting
  $x=\Phi(\ell_1),y=\Phi(\ell_1'),z=\Phi(\ell_2),t=\Phi(\ell_2')$ we
  have that $xy+zt=0$ in $A^2$. But
  $A$ being a Sidon space implies, according to the remark following
  \eqref{eq:sidon}, that either $x$ is an $F$-multiple of $z$ and $y$
  is an $F$-multiple of $t$ or $x$ is an $F$-multiple of $t$ and $y$
  is an $F$-multiple of $z$. Since the mapping $\Phi$ is one-to-one
  from $\LL_n$ to $A$ we deduce from this that  $Q$ is $F$-proportional to $\ell_1\ell_1'$, 
 which   contradicts $\wt(Q)=2$. The reverse implication should be obvious.
\end{proof}

\paragraph{\it Example of a Sidon Space.} Let $\Q$ denote the rational
field and consider the extension $L=\Q[x,y]/(x^2+y^2+1)$. 
The polynomial $x^2+y^2+1$ is absolutely irreducible over $\Q$,
meaning that $L$ contains no proper finite intermediate extension of $\Q$.
We claim that
the subspace $A$ with basis $(1,x,y)$ is a Sidon space such that
$\dimn A^2=5=2\dimn A-1$. The statement regarding the dimensions is
obvious: to see that $A$ is indeed a Sidon space, suppose the contrary
which would mean that four $\Q$-linear combinations $a,b,c,d$ of
$1,x,y$ are such that $ab+cd=x^2+y^2+1$. Switching to projective
coordinates, this would mean exactly that the quadratic form $x^2+y^2+z^2$ is
of weight $2$ over $\Q$. But it is readily seen that this form is
actually of weight $3$ over $\Q$, therefore $A$ can only be a Sidon
space.

As remarked at the beginning of the section, the definition of a Sidon
space contradicts, for dimensions at least $3$, having a basis in
geometric progression. Therefore the above example is also an example
of a space $A$ satisfying $\dimn A^2=2\dimn A-1$ which does not have
a basis in geometric progression. The setting is that of an infinite
extension but this Sidon space can be transplanted as follows into a finite 
 extension without intermediate subfields. 

\paragraph{\it Example of a Sidon Space in a finite extension.}
Let $L=\Q(a)$ where $a$ is a root of the irreducible polynomial
$P(x)=2+2x+2x^2+2x^4+2x^5+x^8 = 1+x^2+(1+x+x^4)^2$. It can be checked
that the Galois group of this polynomial is the full symmetric group
$S_8$
which implies that the extension $L/\Q$ has no intermediate extension.
Setting $A$ to be the subspace of $L$ generated by $1,a,b=1+a+a^4$, we
again have $\dimn A^2=5$ since $1+a^2+b^2=0$. 
We also have that if $\Phi$ is the mapping from the space of quadratic
forms in the variables $x,y,z$ over $\Q$ into $A$ defined by $x\mapsto
a$, $y\mapsto b$, $z\mapsto 1$, then $\ker\Phi$ has dimension $1$ and
is generated by the quadratic form $x^2+y^2+z^2$ which is of weight $3$
as mentioned above. Therefore $A$ is a Sidon set in $L$.

We remark that the quadratic form $x^2+y^2+z^2$ is anisotropic over
$\Q$, and that this makes it of weight $3$. Similar examples can be
constructed whenever we have such a quadratic form: these forms do not exist over finite fields
however, which motivates our quest for
Theorem~\ref{thm:main} and our study of packings of quadratic forms of
minimum weight $3$ in the next section.

Finally, note that if $\dimn A^2=2\dimn A-1=2n-1$, we have 
$\dim(\ker\Phi) =n(n+1)/2-(2n-1) =
(n-1)(n-2)/2$. Theorem~\ref{thm:sidon} will therefore follow if we can prove
that for every $n\geq 3$, a subspace of $\QQ_n$ of dimension
$(n-1)(n-2)/2$ and of minimum weight $3$ does not exist.

\section{Codes in the space of quadratic forms over finite fields}\label{sec:codes}

In this section, $q$ is a power of a prime and $F=\F_q$ is the finite field with $q$ elements.
As mentioned in the previous section, we will prove the following theorem, which will
enable us to conclude the proof of Theorem~\ref{thm:main}.

\begin{theorem}\label{Theorem main}
Let $n\geq 3$, and let $\CC\subset \QQ_n$, with  $\wt(\CC)\geq 3$. Then,
$|\CC| <  q^{(n-1)(n-2)/2}$.
\end{theorem}

The inequality $|\CC| \leq   q^{(n-1)(n-2)/2}$ follows from a simple
packing argument.
Indeed, if  $\ell_1$ and $\ell_2$ are two linearly independent elements of $\LL_n$, let
\begin{equation}\label{eq A}
\A_{\ell_1,\ell_2}:=\{ Q=\ell_1(x)\ell_1'(x)+\ell_2(x)\ell_2'(x)\ :\  \ \ell'_i\in \LL_n\}\subset \QQ_n.
\end{equation}
$\A_{\ell_1,\ell_2}$ is a subspace of $\QQ_n$, of dimension $2n-1$: if $\ell_3,\dots,\ell_n$ are such that   $\{\ell_1,\ell_2,\dots,\ell_n\}$ 
is a basis of $\LL_n$, then $\{\ell_1(x)\ell_i(x), \ell_2(x)\ell_j(x), 1\leq i\leq n, 2\leq j\leq n\}$ is a basis of $\A_{\ell_1,\ell_2}$.
Moreover, for all $Q\in \A_{\ell_1,\ell_2}$,  $\wt(Q)\leq 2$. 
So, if $\CC$ has minimum weight at least $3$, and if $(Q,Q')\in \CC^2$, then 
$Q-Q'\notin \A_{\ell_1,\ell_2}$ unless $Q=Q'$. Said differently,  the
elements of $\CC$ are pairwise distinct in the quotient space
$\QQ_n/\A_{\ell_1,\ell_2}$. As a consequence, their cardinalities satisfy
\begin{equation}\label{bound}
|\CC|\leq |\QQ_n/\A_{\ell_1,\ell_2}|=q^{(n-1)(n-2)/2}.
\end{equation}
The bound \eqref{bound} is an instance of an {\em anticode} bound in
coding theory terminology \cite[Ch 17]{mcws}: the set
$\A_{\ell_1,\ell_2}$ is an anticode, i.e. a set of diameter
$2$. Henceforth we refer to $\CC$ as a {\em code}, i.e. a set
$\CC\subset \QQ_n$ with a minimum weight condition.

The claim in Theorem \ref{Theorem main} is therefore that the inequality in
\eqref{bound} is not attained when $n\geq 3$. Improving upon
\eqref{bound} seems out of reach by elementary packing arguments and
we shall need to involve group actions. As mentioned in the
introduction, we shall derive Theorem \ref{Theorem main} by applying the
Delsarte ``linear programming'' method.

Let us first give a short sketch of the proof. The space $\QQ_n$ affords the action of the group $G:=\F_q^*\times \Gl(n,q)$, where $\F_q^*$ acts by scalar multiplication on quadratic forms and the linear group $\Gl(n,q)$ acts linearly on the variables. This action preserves 
$\LL_n$, and hence the weight of a quadratic form. The orbits of the action of $G$ on $\QQ_n$ will be denoted $\{\OO_t,\ t\in \TT_n\}$, where
the index set $\TT_n$ will be specified later. For a code of
quadratic forms  $\CC$, let $X_t:=|\CC\cap
\OO_t|$. In a first step, we will prove that, under the assumptions
$\wt(\CC)\geq 3$ and $|\CC|=q^{(n-1)(n-2)/2}$, these numbers satisfy a
certain system of linear equations (Proposition~\ref{prop:primal system}). Some of the equations are
straightforward translations of the assumptions, but the most
interesting ones arise from an application of Delsarte linear
programming method \cite{Del}, that focuses on an appropriate test function
with non negative Fourier coefficients.  In a second step,  we will
prove that this system doesn't have any real solutions when $n\geq 3$. To this end, we will
work on  a much smaller linear system (Proposition~\ref{prop eq all
  dual}) satisfied by certain variables
related to the Fourier
coefficients of the characteristic function of $\CC$. When the code
$\CC$ is linear, these new variables can be interpreted in terms of
the dual code of $\CC$. The duality notion that we shall rely on is
somewhat non-standard and associates to a code of quadratic forms a
dual code of symmetric bilinear forms.

We start by recalling some classical results on the classification of
quadratic forms and symmetric bilinear forms over finite fields.

\subsection{Group actions on quadratic forms and symmetric bilinear forms over finite
  fields}\label{subsection classification}

Let  $x=(x_1,\ldots ,x_n)$, and
$u=(u_{i,j})_{1\leq i,j\leq n}\in \Gl(n,q)$. Denoting by $u^t$ the
transpose matrix of $u$, to any
quadratic form 
$Q: x\mapsto Q(x)$ we can associate the form $Q^u: x\mapsto Q(xu^t)$,
\begin{equation*}
Q^u(x)=Q\left(\sum_{k=1}^n u_{1,k} x_k,\dots, \sum_{k=1}^nu_{n,k} x_k\right).
\end{equation*}
For $g=(a,u)\in G=\F_q^*\times \Gl(n,q)$, a (right) group action is thus
defined on elements $Q$ of $\QQ_n$ by $Q^g=aQ^u$.
The orbits of $\QQ_n$ under this action
are straightforwardly obtained from the well known description of the
orbits of $\QQ_n$  under $\Gl(n,q)$ (\cite[Chapter 11]{Tay}); we
explicitly describe them in the next proposition, where we
use the notation $Q\sim Q'$ for two $G$-equivalent quadratic forms. 

Let us first recall the notion of \emph{rank} of a quadratic form:  the
rank of $Q$ is the codimension of its \emph{radical} $\Rad_Q$,  the linear space of elements $x\in\F_q^n$ such that $Q(x)=0$
and $B_Q(x,y)=0$ for all $y\in \F_q^n$, where
$B_Q(x,y)=Q(x+y)-Q(x)-Q(y)$ is the symmetric bilinear form associated
to $Q$. A quadratic form  $Q\in \QQ_n$ is said to be non degenerate
if $\rk(Q)=n$. 

Over finite fields, the $G$-orbit of a quadratic form is not
characterised solely by it rank. It turns out that the set  of quadratic forms of given  rank $r$ makes up
one orbit when $r$ is odd, but splits into two orbits when $r$ is even:

\begin{proposition}\label{proposition equivalence Q} If  $Q\in \QQ_n$ is a quadratic form over $\F_q$ in $n$ variables
  of rank $r=\rk(Q)>0$, then one of the following holds:
\begin{enumerate}
\item[(1)] $r$ is odd and $Q\sim \sum_{i=1}^{(r-1)/2} x_{2i-1}x_{2i}+ x_r^2$. 
\item[(2)] $r$ is even and $Q\sim \sum_{i=1}^{r/2} x_{2i-1}x_{2i}$. 
\item[(3)] $r$ is even and $Q\sim \sum_{i=1}^{r/2-1} x_{2i-1}x_{2i}+
  Q_{0}(x_{r-1},x_r)$, where 
\begin{equation*}
Q_{0}(x_1,x_2)=\begin{cases}  x_1^2- bx_2^2\text{  if }p\neq 2\\
x_1^2+x_1x_2+bx_2^2\text{  if }p=2.
\end{cases}
\end{equation*}
In the above, $b\in \F_q$ is such that 
$b\in \F_q\setminus \F_q^2$ if $p\neq 2$, and   $b\in \F_q\setminus\sigma(\F_q)$ if $p=2$, where
  $\sigma(\F_q)=\{\lambda^2+\lambda,\ \lambda\in \F_q\}$.
\end{enumerate}
\end{proposition}

So, a quadratic form over $\F_q$ is either  equivalent to a hyperbolic
form (this is case (2)), or to the direct sum of a hyperbolic form and
of an anisotropic form, i.e. a quadratic form without non trivial
zeroes,  in one (case (1)) or two (case (3)) variables. If $Q$ is a
quadratic form of rank $r$, and if $e\in \{0,1,2\}$ denotes the rank
of its anisotropic component, we will say that $Q$ has \emph{type}
$(r,e)$ and we will write $Q\sim (r,e)$. 

Let us now make the connection between the weight and
the type of a quadratic form. Since the weight of a quadratic form
is left unchanged under the action of $G$, it should be expressible in terms
of its type. Indeed, we have: 

\begin{lemma}\label{lemma weight to type} If $Q\sim (r,e)$, 
$\wt(Q)=(r+e)/2$.
\end{lemma}

\begin{proof} If $Q$ is hyperbolic of rank $2h$,
  meaning that  $Q\sim \sum_{i=1}^{h} x_{2i-1}x_{2i}$, the weight of
  $Q$ is cleary equal to $h$. So, the formula is verified when
  $e=0$. When $e=1$, because $\wt(x_r^2)=1$, we get
  $\wt(Q)=(r-1)/2+1=(r+1)/2$.
When $e=2$, it is immediate that,  with the notation of Proposition
\ref{proposition equivalence Q}, $\wt(Q_0)=2$, so $\wt(Q)=(r/2-1)+2=(r+2)/2$.
\end{proof}

The space $\SSS_n$ of symmetric bilinear forms over $\F_q$ is also
equipped with a natural action of $G$, given by $B^g=au^tBu$, where we
identify symmetric bilinear forms and symmetric matrices. We note that
the $\F_q$-vector spaces $\QQ_n$ and $\SSS_n$ have the same dimension
$n(n+1)/2$. When $q$ is odd, the
correspondence $Q\mapsto B_Q$ defines an  isomorphism of $G$-spaces,
 because $Q$ can be recovered from $B_Q$ thanks to the
formula $B_Q(x,x)=2Q(x)$. But when  $q$
is even,  $B_Q$ is alternating and the correspondence $Q\mapsto B_Q$ is
not one to one. However, for all $q$,  there exists a \emph{non
  degenerate pairing}
between $\QQ_n$ and $\SSS_n$ that behaves well with respect to the
action of $G$. Let us recall that a pairing $(\ ,\ ): \QQ_n\times
\SSS_n \to \C^*$ is an application which is homomorphic with respect
to $Q$ and $B$, i.e. $(Q+Q',B)=(Q,B)(Q',B)$ and $(Q,B+B')=(Q,B)(Q,B')$ for all $Q,Q'\in \QQ_n$,
$B,B'\in \SSS_n$ and that it is said to be non degenerate if
$(Q,B)=1$ for all $B\in \SSS_n$ implies that $Q=0$, and similarly for
$B\in \SSS_n$. We note that the duality between the association schemes of quadratic forms and of symmetric bilinear forms derived from this pairing was already observed in \cite{WWMM}.

\begin{lemma}\label{lemma pairing}
Let $\alpha:(\F_q,+)\to (\C^*,\times)$ be a fixed non trivial character. Let, for
$Q\in \QQ_n$, $Q=\sum_{1\leq i\leq j\leq n} a_{i,j} x_i x_j$ and
  $B\in \SSS_n$, $B=(b_{i,j})_{1\leq i,j\leq n}$, 
\begin{equation}\label{def pairing}
(Q,B):=\alpha\big(\sum_{1\leq i\leq j\leq n} a_{i,j}  b_{i,j} \big).
\end{equation}
Then, this expression defines a non degenerate pairing between $\QQ_n$ and
$\SSS_n$. Moreover, for all $Q\in \QQ_n$, $B\in \SSS_n$, $g\in G$, we have 
\begin{equation}\label{pairing action}
(Q^g, B)=(Q,B^{g^t}) 
\end{equation}
where, if $g=(a,u)\in G$, we denote  $g^t=(a,u^t)$.
\end{lemma}

\begin{proof} It is a straightforward verification.
\end{proof}

The pairing introduced above allows for a convenient description of
the multiplicative characters of the additive group $(\QQ_n,+)$ and of
the way the group $G$ acts on them. Indeed, the characters of
$(\QQ_n,+)$ are in one to one correspondence with $\SSS_n$ and for
every $B\in \SSS_n$ the associated character is given by:
\begin{equation}
\chi_B(Q)=(Q,B)
\end{equation}
where $B\in \SSS_n$. Furthermore, if we define the action of $G$ on
characters by $(g.\chi)(Q):=\chi(Q^g)$, then Lemma \ref{lemma pairing}
translates into:
$$g.{\chi_B}=\chi_{B^{g^t}}.$$

In the next proposition, we recall the description of the orbits of $\SSS_n$ under
the action of $G$. If $q$ is even, and $B\in \SSS_n$ is non degenerate, $W:=\{x\in
\F_q^n\ :\ B(x,x)=0\}$ is a hyperplane, and the restriction of $B$ to
$W$ is alternating, so the description of the $\Gl(n,q)$-orbits of $\SSS_n$ follows
easily from the classification of alternating forms \cite[Chapter
8]{Tay}.  A matrix with a diagonal block
structure $(\begin{smallmatrix}A&0\\0&B\end{smallmatrix})$ is denoted
below by $A\oplus B$. 

\begin{proposition}\label{proposition equivalence S} If  $B\in \SSS_n$
  is a symmetric bilinear  form over $\F_q$ in $n$ variables
  of rank $\rk(B)=r>0$, then one of the following holds:
\begin{enumerate}
\item[(1)] $r$ is odd and $B\sim \oplus_{i=1}^{(r-1)/2}  (\begin{smallmatrix}0&1\\1&0\end{smallmatrix})\oplus (1)$. 
\item[(2)] $r$ is even and $B\sim \oplus_{i=1}^{r/2}  (\begin{smallmatrix}0&1\\1&0\end{smallmatrix})$. 
\item[(3)] $r$ is even and 
$B\sim
\oplus_{i=1}^{r/2-1}(\begin{smallmatrix}0&1\\1&0\end{smallmatrix})
\oplus B_0$
where $B_0=(\begin{smallmatrix}0&1\\1&1\end{smallmatrix})$ if $p=2$,
and $B_0=(\begin{smallmatrix}1&0\\0&-b\end{smallmatrix})$,  $b\in \F_q\setminus \F_q^2$, if $p\neq2$.
\end{enumerate}
\end{proposition}

Similarly to quadratic forms, we say that a symmetric bilinear
form $B$ is of \emph{type}
$(r,e)$, and write $B\sim (r,e)$,  when $r$ is the rank of $B$, and
$e$ takes the value $e=1$, $0$, $2$ when $B$ is in case (1), (2), (3), respectively. 
If in addition we define the type of $Q=0$, and of $B=0$, to be $0$, 
we obtain a convenient parametrization of the orbits of $\QQ_n$ and of $\SSS_n$
under the action of $G$ by the same set $\TT_n$:
\begin{equation}
\TT_n:=\{0\}\cup \{ (r,e)\ : \ 1\leq r\leq n,\ e\in \{0,1,2\},\ r\equiv
e\bmod 2,\ r\geq e\}.
\end{equation}

When dealing with types, orbits or stabilizers, we will use the same
set of notations regardless of whether we are speaking of quadratic forms or of
symmetric bilinear forms. However, to avoid confusion, we will
preferentially reserve the letter $t$ for types of elements of $\QQ_n$
and the letter $s$ for types of elements of $\SSS_n$.

\subsection{The equations satisfied by $\{X_t, \ t\in
  \TT_n\}$}\label{subsection equations}

Let $\CC\subset \QQ_n$ denote a (not
necessarily linear) subset of $\QQ_n$, such that $\wt(\CC)\geq 3$.
Without loss of generality, we will assume  that $0\in \CC$. We will
be concerned with the unknowns $X_t$, for $t\in \TT_n$:
\begin{equation}\label{inc}
X_t:=|\CC\cap \OO_t|=|\{Q\in \CC \ :\ Q \sim t\}|.
\end{equation}
These numbers satisfy a few trivial equations:
\begin{equation}\label{eq-trivial}
\begin{cases}
X_0=1\\
\sum_{t\in \TT_n} X_t=|\CC|\\
X_t=0 \text{ for } t=(1,1), (2,0),(2,2),(3,1),(4,0).
\end{cases}
\end{equation}
The last set of conditions follow from the assumption $\wt(\CC)\geq
3$, which implies that $X_t=0$ when $t$ is the type of an orbit of
weight $1$ or $2$, and from Lemma \ref{lemma weight to type}. We will
prove in the next proposition that, if $\CC$ is
a hypothetical optimal code, i.e. if it satisfies equality in
\eqref{bound}, then some additional equations hold for
$\{X_t\}$. First we need to introduce another notation. 

We make the remark that, for $s\in \TT_n$, the expression 
$\sum_{B\in \OO_s} \chi_B(Q)=\sum_{B\in \OO_s} (Q,B)$ only depends on  the type of
$Q$. Indeed, this property follows immediately from \eqref{pairing
  action}. This remark justifies the following notation:
\begin{equation}\label{def chis}
\chi_s(t):= \sum_{B\in \OO_s} (Q,B) \qquad (Q\sim t).
\end{equation}

\begin{proposition}\label{prop extra eq}
Let $\CC\subset \QQ_n$,  containing  $0$, such that $\wt(\CC)\geq
3$  and $|\CC|=q^{(n-1)(n-2)/2}$. For all $s=(r,e)$ such that $1\leq r\leq n-2$, we have
\begin{equation}\label{eq not used}
\sum_{t\in \TT_n} \chi_{s}(t) X_t=0.
\end{equation}
\end{proposition}
 
\begin{proof}
We consider the function on $\QQ_n$ defined by:
\begin{equation}
F(Q):=\sum_{g\in G} \1_{\A^g}(Q)
\end{equation}
where $\A:=\A_{x_1,x_2}$. We note that $\A^{g}$ runs over the set
$\{\A_{\ell_1,\ell_2}\ :\  (\ell_1,\ell_2) \text{ independent}\}$. 
We will compute in two different ways the expression:
\begin{equation}\label{sum}
\Sigma:=\sum_{(Q,Q')\in \CC^2} F(Q-Q').
\end{equation}

The first way is straightforward; if $Q\neq Q'$, the weight of $Q-Q'$
is at least equal to $3$, consequently  $Q-Q'$ cannot
belong to $\A^g$, so 
\begin{equation}\label{first}
\Sigma=\sum_{Q\in \CC} F(0)= |G| |\CC|.
\end{equation}

The second method uses the expansion of $F$ over the characters of
$\QQ_n$, in other words its  Fourier expansion. We introduce the complex vector space
\begin{equation*}
L(\QQ_n):=\{f\ :\ \QQ_n\to \C\}
\end{equation*}
The
multiplicative characters $\chi_B:(\QQ_n,+)\to \C^*$, for $B\in \SSS_n$,  form a basis of
$L(\QQ_n)$, which is orthonormal for the standard inner product
\begin{equation*}
\langle f_1,f_2\rangle:=\frac{1}{|\QQ_n|}\sum_{Q\in \QQ_n} f_1(Q)\overline{f_2(Q)}
\qquad f_1,f_2\in L(\QQ_n).
\end{equation*}

So, $F$ can be written  $F=\sum_{B\in \SSS_n} f_{B} \chi_B$. It will be
essential in what follows that $f_B\geq 0$, and we will also need to
know when $f_B>0$. The next lemma clarifies this.

\begin{lemma}\label{lemma fB}
With the above notation, $f_B\geq 0$ for all $B\in \SSS_n$, and
$f_B>0$ if and only if $\rk(B)\leq n-2$.  Moreover,
$f_0=|G||\A|/|\QQ_n|$.
\end{lemma}

\begin{proof} We have
\begin{align*}
f_B=\langle F, \chi_B\rangle &=\frac{1}{|\QQ_n|}\sum_{Q\in \QQ_n}
F(Q) \overline{\chi_{B}(Q)}\\
&=\frac{1}{|\QQ_n|}\sum_{Q\in \QQ_n}\Big(\sum_{g\in G}
\1_{\A^g}(Q) \overline{\chi_{B}(Q)}\Big)\\
&= \frac{1}{|\QQ_n|}\sum_{g\in G}\Big( \sum_{Q\in
  \A^g}\overline{\chi_{B}(Q)}\Big).
\end{align*}
Let $\A^{\perp}=\{B\in \SSS_n\ :\ (Q,B)=1\text{ for all }Q\in
\A\}$. Then, from the property \eqref{pairing action}  of the pairing,
$(\A^g)^{\perp}=(\A^{\perp})^{g^{-t}}$ (where $g^{-t}$ denotes the
transpose of the inverse of $g$), and, because $\A^g$ is a subgroup of
$\QQ_n$,
\begin{equation*}
\sum_{Q\in \A^g} \chi_B(Q)=\left\{
\begin{array}{ll}
0 & \text{if }B\notin (\A^\perp)^{g^{-t}}\\
|\A^g|=|\A| & \text{if }B\in (\A^\perp)^{g^{-t}}.
\end{array}
\right.
\end{equation*}
So, we find
\begin{equation}\label{equation fB}
f_B=\frac{|\A|}{|\QQ_n|}|\{g\in G\ :\ B^{g^t}\in \A^{\perp}\}|.
\end{equation}
From \eqref{equation fB}, it is clear that $f_B\geq 0$. Moreover, we see that
$f_B>0$ if and only if the orbit of $B$ intersects $\A^{\perp}$. Going
back to the definitions of $\A$ \eqref{eq A} and of the pairing \eqref{def pairing}, we see that
$\A^{\perp} =\{ B'\in \SSS_n \ :\ B'_{1,j}=B'_{2,j}=0\text{ for all }
1\leq j\leq n\}$. So, the orbit of
$B$ intersects $\A^{\perp}$ if and only if $\rk(B)\leq n-2$. The
expression for $f_0$ follows from \eqref{equation fB}.
\end{proof}

Going back to $\Sigma$ in \eqref{sum}, we have:
\begin{equation}\label{second}
\Sigma =\sum_{B\in \SSS_n} f_{B} \big(\sum_{(Q,Q')\in \CC^2} \chi_B(Q-Q')\big)
= \sum_{B\in \SSS_n} f_{B} \big\vert\sum_{Q\in \CC} \chi_B(Q)\big\vert^2
\geq f_0 |\CC|^2
\end{equation}
where, in the last inequality, we have neglected the contributions of
all the characters except that of the trivial one; the non negativity
of the coefficients $f_{B}$ is crucial in this step.

Now we recap our findings: from \eqref{first}, \eqref{second} and
Lemma \ref{lemma fB},

\begin{equation*}
|G||\CC|=\Sigma\geq \frac{|G||\A|}{|\QQ_n|}|\CC|^2.
\end{equation*}
So, we obtain after simplification the inequality $1\geq
(|\A|/|\QQ_n|)|\CC|$; this inequality is nothing more than \eqref{bound}. The
interesting point here is that  the equality $1=
(|\A|/|\QQ_n|)|\CC|$ holds if and only if 
the neglected terms in \eqref{second} are equal to zero, leading, in the
case of a hypothetical optimal code, to the conditions
\begin{equation}\label{eq new}
\sum_{Q\in \CC} \chi_B(Q)=0 \qquad \text{ if } f_{B}>0,\ B\neq 0.
\end{equation}

From Lemma \ref{lemma fB}, we know that $f_B>0$ if and only if
$\rk(B)\leq n-2$. Let $s=(r,e)$ where $1\leq r\leq n-2$; 
observing that the condition  $f_B>0$ holds
simultaneously for all the elements of the orbit $\OO_s$ allows us to sum up the 
equations \eqref{eq new} over $B\in \OO_{s}$. The expression 
$\sum_{B\in \OO_s} \chi_B(Q)=\sum_{B\in \OO_s} (Q,B)$ only depends on
the type of $Q$, since
\begin{equation*}
\sum_{B\in \OO_s} (Q^g,B) = \sum_{B\in \OO_s} (Q,B^{g^t})= \sum_{B'\in \OO_s} (Q,B'),
\end{equation*}
so the definition \eqref{def chis} of $\chi_s(t)$ is consistent and we
obtain that
\begin{equation*}
\sum_{t\in \TT_n} \chi_s(t)X_t=\sum_{B\in \OO_s} \big(\sum_{Q\in \CC} \chi_B(Q)\big)=0.
\end{equation*}
\end{proof}

\subsection{A change of variables}\label{subsection change}

Let us recapitulate what we have achieved by now:

\begin{proposition}\label{prop:primal system}
Suppose $\CC\subset \QQ_n$ is such that $0\in \CC$, 
$|\CC|=q^{(n-1)(n-2)/2}$ and $\wt(\CC)\geq~3$. Then,
\begin{equation*}
X_t:= |\CC \cap \OO_t| \quad (t\in \TT_n)
\end{equation*}
satisfy the following equations:

\begin{equation}\label{eq all}
\left\{
\begin{array}{ll}
(i) &X_0=1\\
(ii) &X_t=0 \quad \text{ for }  t=(1,1), (2,0),(2,2),(3,1),(4,0)\\
(iii) &\sum_{t\in \TT_n} X_t=q^{(n-1)(n-2)/2}\\
(iv) &\sum_{t\in \TT_n} \chi_s(t) X_t=0 \quad \text{ for all } s=(r,e),\ 1\leq r\leq n-2
\end{array}
\right.
\end{equation}
\end{proposition}

\eqref{eq all} is a linear system of equations with $|\TT_n|$ unknowns (the variables $X_t$)
and $|\TT_n|+3$ equations. So, if we are reasonably lucky, this linear
system has no solutions, which in turn proves the non existence of the code $\CC$. Unfortunately, only a few of the unknowns are equal to zero so we cannot easily reduce the size of this system, and moreover,  we so far have no closed expression for the coefficients $\chi_s(t)$.
In order to overcome these issues, we will introduce a change of variables, whose effect will be to `exchange' 
variables and equations, the end result yielding a very small linear system.

\begin{proposition}\label{prop eq all dual}
Suppose $\CC\subset \QQ_n$ is such that $0\in \CC$, 
$|\CC|=q^{(n-1)(n-2)/2}$ and $\wt(\CC)\geq~3$. Let
\begin{equation*}
X_t:=|\CC \cap \OO_t| \quad (t\in \TT_n)
\end{equation*}
and let 
\begin{equation*}
Y_s:= \frac{1}{|\CC|}\sum_{t\in \TT_n} \chi_s(t)X_t \quad (s\in \TT_n).
\end{equation*}
The numbers $Y_s$  satisfy the following equations:
\begin{equation}\label{eq all dual}
\left\{
\begin{array}{ll}
(i') & Y_0=1\\
(ii') &  Y_s=0 \quad \text{ for all } s=(r,e),\ 1\leq r\leq n-2\\
(iii')& \sum_{s\in \TT_n} Y_s=q^{2n-1}\\
(iv') &\sum_{s\in \TT_n} \frac{\chi_s(t)}{|\OO_s|} Y_s=0 \quad  t=(1,1), (2,0),(2,2),(3,1),(4,0)\\
\end{array}
\right.
\end{equation}
 \end{proposition}

\begin{proof} The new variables $Y_t$ are related to the coefficients $\{\lambda_B,\ B\in \SSS_n\}$
  of the characteristic function $\1_\CC$ of $\CC$ on the basis of
  characters:
\begin{equation}\label{eq 1C}
\1_{\CC}=\sum_{B\in \SSS_n} \lambda_B \chi_B.
\end{equation}
Indeed, we have
\begin{equation}\label{eq lambdaB}
\lambda_B=\langle \1_\CC , \chi_B \rangle=\frac{1}{|\QQ_n|}\sum_{Q\in  \CC}\overline{\chi_B(Q)},
\end{equation}
so 
\begin{equation*}
\sum_{B\in \OO_s} \lambda_B
=\frac{1}{|\QQ_n|}\sum_{Q\in  \CC}\Big(\sum_{B\in \OO_s}
(Q,-B)\Big)=\frac{1}{|\QQ_n|}\sum_{t\in \TT_n} \chi_s(t) X_t
\end{equation*}
and thus
\begin{equation}\label{eq Ys}
Y_s=\frac{|\QQ_n|}{|\CC|}\sum_{B\in \OO_s} \lambda_B.
\end{equation}

Let us verify the equations $(i')-(iv')$. 
From \eqref{eq Ys}, $Y_0=\frac{|\QQ_n|}{|\CC|} \lambda_0$ and from
\eqref{eq lambdaB}, $\lambda_0=\frac {|\CC|}{|\QQ_n|}$
so we find $(i')$. The equations $(ii')$ follow immediately from
$(iv)$. From \eqref{eq lambdaB} and \eqref{eq 1C}, we obtain $(iii')$:
\begin{equation*}
\sum_{s\in \TT_n} Y_s= \frac{|\QQ_n|}{|\CC|}\sum_{B\in \SSS_n}
\lambda_B=\frac{|\QQ_n|}{|\CC|}\1_\CC(0)=q^{2n-1}.
\end{equation*}

It remains to prove $(iv)'$. To this end, we introduce the
characteristic function $\1_{\OO_t}$ of the orbit $\OO_t$ of quadratic
forms of type $t$ and its decomposition as a linear combination of
characters:
\begin{equation*}
\1_{\OO_t}=\sum_{B\in \SSS_n} \mu_{t,B}\chi_B
\end{equation*}
where
\begin{equation*}
\mu_{t,B}=\langle \1_{\OO_t},\chi_B\rangle=\frac{1}{|\QQ_n|}\sum_{Q\in
  \OO_t} \overline{\chi_B(Q)} = \frac{1}{|\QQ_n|}\sum_{Q\in\OO_t}\chi_B(Q).
\end{equation*}
The last equality holds because for any $Q\in\OO_t$ we have
$-Q\in\OO_t$ and $\chi_B(-Q)=(-Q,B)=\overline{(Q,B)}$ by definition
\eqref{def pairing} of $(Q,B)$.
Now the expression $\sum_{Q\in  \OO_t} \chi_B(Q)$ only depends
on the type $s$ of $B$, and 
\begin{align*}
|\OO_s|\big(\sum_{Q\in  \OO_t}\chi_B(Q)\big)&=
\sum_{B\in \OO_s}\big(\sum_{Q\in  \OO_t} \chi_B(Q)\big)\\
&=\sum_{Q\in  \OO_t} \big(\sum_{B\in \OO_s}\chi_B(Q)\big)
=\sum_{Q\in  \OO_t} \chi_s(t)=|\OO_t| \chi_s(t)
\end{align*}
so
\begin{equation}\label{eq mutB}
\mu_{t,B}=\frac{|\OO_t|\chi_s(t)}{|\QQ_n||\OO_s|}.
\end{equation}
Then,
\begin{equation*}
X_t=|\QQ_n|\langle \1_{\OO_t},\1_\CC \rangle=|\QQ_n|\sum_{B\in \SSS_n}
\mu_{t,B} \lambda_B.
\end{equation*}
Taking account of \eqref{eq mutB} and of \eqref{eq Ys}, we find
\begin{align*}
X_t&=|\OO_t| \sum_{B\in \SSS_n} \frac{\chi_s(t)}{|\OO_s|} \lambda_B 
=|\OO_t| \sum_{s\in \TT_n} \frac{\chi_s(t)}{|\OO_s|}\big(\sum_{B\in \OO_s} \lambda_B \big)\\
&=\frac{|\OO_t| |\CC|}{|\QQ_n|} \sum_{s\in \TT_n}
\frac{\chi_s(t)}{|\OO_s|} Y_s.
\end{align*}
The above relation between $X_t$ and $Y_s$ shows that $(ii)$ is
equivalent to $(iv')$.
\end{proof}

If the code $\CC$ is a linear subspace of $\QQ_n$, the numbers $Y_s$ have a combinatorial
interpretation in terms of the dual code 
\begin{equation*}
\CC^{\perp}=\{B\in \SSS_n\
:\ (Q,B)=1\text{ for all }Q\in \CC\}.
\end{equation*}
Indeed, it follows from the Poisson summation formula that
$Y_s=|\CC^\perp \cap \OO_s|$. Moreover, 
it now follows from Proposition \ref{prop eq all dual} that we have a remarkable
connection between linear codes of quadratic forms endowed with the
weight distance, and linear codes of symmetric bilinear forms endowed
with the rank distance, namely:
\begin{corollary}\label{cor:linear}
  If $\CC$ is a linear subspace of $\QQ_n$ 
of dimension $(n-1)(n-2)/2$ and minimum weight $3$, then $\CC^{\perp}$ is a linear subspace of $\SSS_n$ 
of dimension $2n-1$ and minimum rank $n-1$. 
\end{corollary}

Codes of symmetric bilinear forms for the rank distance have been studied in \cite{Sch1} and \cite{Sch2}.
In particular, the tight upper bound of $q^{n+1}$ is proved for a linear code of minimum rank at least $n-1$,
if $q$ is even and $n$ is even (\cite{Sch1}), and  if $q$ is odd
(\cite{Sch2}). Noting that $n+1<2n-1$ if and only if $n\geq 3$, we
have therefore that Corollary~\ref{cor:linear} together with
\cite{Sch2} proves Theorem~\ref{thm:sidon} for $q$ odd and hence
Theorem~\ref{thm:main} for $q$ odd. For even $q$, the results of
\cite{Sch1} together with Corollary~\ref{cor:linear} yield
Theorem~\ref{thm:sidon} with the additional hypothesis that $\dim(A)$
is even. If we remark that Lemma~\ref{lem:dim(A)even} proves that
under the hypothesis of Theorem~\ref{thm:main} the
$2$-atom $A$ cannot have odd dimension, then we obtain a proof of
Theorem~\ref{thm:main} for finite fields with $q$ even. 

In the next section we shall nevertheless give a self-contained proof
that the system \eqref{eq all dual} has no solutions over the real
numbers.
This will give us a complete proof of Theorem~\ref{thm:sidon} which is
of independent interest besides being a stepping stone towards 
Theorem~\ref{thm:main}.

\subsection{Solving the linear system satisfied by $\{Y_s,\ s\in
  \TT_n\}$}

The linear system \eqref{eq all dual} is much more friendly than
\eqref{eq all} because it can be re-written as a linear system in
only three variables (either
$\{Y_{(n-1,1)}, Y_{(n,0)}, Y_{(n,2)}\}$ or $\{Y_{(n-1,0)},
Y_{(n-1,2)}, Y_{(n,1)}\}$, depending on the parity of $n$) and six equations. 
In order to prove that this linear system doesn't have any solutions
when $n\geq 3$, it will be enough to take account of the equations
$(iv')$ associated to the types $t=(1,1),
(2,0),(2,2)$, and thus to compute the values of $\chi_s(t)$ for
$s=(n,e), (n-1,e)$ and $t=(1,1),
(2,0),(2,2)$. These numbers are the so-called \emph{$P$-numbers}  of the association scheme 
defined by the action of $G$ on $\QQ_n$ (\cite{BI}) and they could be derived from \cite{FWMM} and \cite{Sch2}.
In order to keep this paper self-contained, we offer a direct  computation in the appendix. We obtain the following systems of equations:

\begin{lemma}\label{lemma system} With the notation of Proposition \ref{prop eq all dual}, let $Y=(Y_{(n-1,1)}, Y_{(n,0)}, Y_{(n,2)})$ if $n$ is even,
  and $Y=(Y_{(n-1,0)}, Y_{(n-1,2)}, Y_{(n,1)})$ if $n$ is odd, 
then
\begin{equation*}
M Y^t=(q^{2n-1}-1, -1,-1,-1)^t
\end{equation*}
where:
\begin{enumerate}
\item If $n$ is even and $p=2$,

\begin{equation*}
M=\left[
\begin{array}{ccc}
1 & 1 & 1\\
\frac{-1}{q^n-1} & 1 & \frac{-1}{q^n-1} \\
\frac{q^n-2q^{n-1}+1}{(q^n-1)(q^{n-1}-1)} & \frac{-1}{q^{n-1}-1}
&\frac{-1}{q^n-1}\\
\frac{-1}{q^{n-1}-1} &\frac{-1}{q^{n-1}-1} &
\frac{q^{n-1}+1}{(q^n-1)(q^{n-1}-1)}
\end{array}
\right]
\end{equation*}

\item If $n$ is odd and $p=2$,

\begin{equation*}
M=\left[
\begin{array}{ccc}
1 & 1 & 1\\
1& \frac{-1}{q^n-1} & \frac{-1}{q^n-1} \\
\frac{-1}{q^n-1} & \frac{q^n-2q^{n-1}+1}{(q^n-1)(q^{n-1}-1)} & \frac{-1}{q^{n}-1}
\\
\frac{-1}{q^{n}-1} & \frac{q^{n}+1}{(q^n-1)(q^{n-1}-1)} & \frac{-1}{q^{n}-1} 
\end{array}
\right]
\end{equation*}

\item If $n$ is even and $p\neq 2$,
\begin{equation*}
M=\left[
\begin{array}{ccc}
1 & 1 & 1\\
\frac{-1}{q^n-1} & \frac{q^{n/2}-1}{q^n-1} & \frac{-q^{n/2}-1}{q^n-1}
\\

\frac{q^n-2q^{n-1}+1}{(q^n-1)(q^{n-1}-1)} & \frac{-q^{n-1}-q^{n/2-1}(q-1)+1}{(q^n-1)(q^{n-1}-1)}
&\frac{-q^{n-1}+q^{n/2-1}(q-1)+1}{(q^n-1)(q^{n-1}-1)}\\
\frac{-1}{q^{n-1}-1} &
\frac{q^{n-1}-q^{n/2-1}(q+1)+1}{(q^n-1)(q^{n-1}-1)}
&\frac{q^{n-1}+q^{n/2-1}(q+1)+1}{(q^n-1)(q^{n-1}-1)}\\
\end{array}
\right]
\end{equation*}

\item If $n$ is odd and $p\neq 2$,
\begin{equation*}
M=\left[
\begin{array}{ccc}
1&1&1\\
\frac{q^{(n+1)/2}-1}{q^n-1} &\frac{-q^{(n+1)/2}-1}{q^n-1}
&\frac{-1}{q^n-1} \\
\frac{q^n-2q^{n-1}+1-q^{(n-1)/2}(q-1)}{(q^n-1)(q^{n-1}-1)} & 
\frac{q^n-2q^{n-1}+1+q^{(n-1)/2}(q-1)}{(q^n-1)(q^{n-1}-1)} & 
\frac{-1}{q^n-1} \\
\frac{q^n+1-q^{(n-1)/2}(q+1)}{(q^n-1)(q^{n-1}-1)} & 
\frac{q^n+1+q^{(n-1)/2}(q-1)}{(q^n-1)(q^{n-1}-1)} & 
\frac{-1}{q^n-1} 
\end{array}
\right]
\end{equation*}

\end{enumerate}
\end{lemma}

The following proposition will conclude the proof of Theorem
\ref{Theorem main}.

\begin{proposition}\label{prop no sol}
The linear systems defined in Lemma \ref{lemma system} have no
solutions if $n\geq 3$. 
\end{proposition}

\begin{proof} By brute force: we solve for the first three equations; it turns out
  that there is a unique solution $Y^*$, and then we evaluate the left
  hand side of the last equation at $Y^*$, and observe that it cannot be
  equal to $-1$ unless $n=1,2$. We skip the details.
\begin{enumerate}
\item If $n$ is even and $p=2$,
\begin{equation*}
Y^*=\Big(\frac{(q^n-1)(q^{n-1}-1)}{q-1}, q^{n-1}-1,
\frac{(q^n-1)(q^n-2q^{n-1}+1)}{q-1}\Big)
\end{equation*}
and the LHS of the last equation evaluated at $Y^*$ is equal to 
\begin{equation*}
-1-2\frac{q^n(q^{n-2}-1)}{(q-1)(q^{n-1}-1)}.
\end{equation*}
\item If $n$ is odd and $p=2$,
\begin{equation*}
Y^*=\Big(q^{n-1}-1, \frac{q(q^{n-1}-1)^2}{q-1}, 
(q^{2n-1}-1)-\frac{(q^{n-1}-1)}{q-1}\Big)
\end{equation*}
and the LHS of the last equation evaluated at $Y^*$ is equal to 
\begin{equation*}
-1+2\frac{q^n(q^{n-1}-1)}{(q-1)(q^{n}-1)}.
\end{equation*}
\item If $n$ is even and $p\neq 2$,
\begin{align*}
Y^*=\Big( \frac{(q^n-1)(q^{n-1}-1)}{q-1}, &
\frac{(q^{2n-1}-1)}{2}+\frac{(q^{n/2}(q-1)-1) (q^{n-1}-1)}{2(q-1)},\\
&\frac{(q^{2n-1}-1)}{2}+\frac{(-q^{n/2}(q-1)-1) (q^{n-1}-1)}{2(q-1)} \Big)
\end{align*}
and the LHS of the last equation evaluated at $Y^*$ is equal to 
\begin{equation*}
-1-2\frac{q^n(q^{n-2}-1)}{(q-1)(q^{n-1}-1)}.
\end{equation*}
\item If $n$ is odd and $p\neq2$,
\begin{align*}
Y^*=\Big( &\frac{(q^{n-1}-1)(q^{(n-1)/2}+1)(q^{(n+1)/2}-1)}{2(q-1)},\\
&\frac{(q^{n-1}-1)(q^{(n-1)/2}-1)(q^{(n+1)/2}-1)}{2(q-1)},\\
&(q^{2n-1}-1)-\frac{(q^{n-1}-1)}{q-1}\Big)
\end{align*}
and the LHS of the last one at $Y^*$ is equal to 
\begin{equation*}
-1+2\frac{q^n(q^{n-1}-1)}{(q-1)(q^{n}-1)}.
\end{equation*}
\end{enumerate}
\end{proof}

\section{Transcendental extensions}\label{sec:transcendental}

We now suppose that the two finite-dimensional subspaces $S$ and $T$ are such that
$\dim_FST =\dim_FS+\dim_FT-1$ but that they live in an infinite
dimensional extension $L/F$, with no element of $L\setminus F$
algebraic over $F$. When $\dimn S,\dimn T\geq 2$, can we conclude that 
$S$ and $T$ have bases in geometric progression as in
Theorem~\ref{thm:main}~?
It is natural to search for guidance in the classical additive
setting: it is much easier to prove
in the set $\Z$ of integers than in $\Z/p\Z$, that if $|S|,|T|\geq 2$ and $|S+T|=|S|+|T|-1$, then $S$ and
$T$ must be arithmetic progressions. This leads us to think that the
case of infinite extensions should be easier than the case of finite
extensions.
This turns out to be true, but only in part. Transcendental extensions
have valuation rings which allow us to transfer the structure of
spaces with small products to sets with small sumsets in totally
ordered abelian groups. This method has limitations however as we
develop below.

We first deal with the case when $F$ is algebraically closed: in this
case the valuation approach is straightforward.

\begin{theorem}\label{thm:closed}
  Let $F$ be an algebraically closed field and let $L/F$ be a non-trivial
  extension of~$F$. Let $S,T$ be subspaces of $L$ such that
  $\dimn S\geq 2,\dimn T\geq 2$ and
  $$\dimn ST=\dimn S+\dimn T-1.$$
  Then $S$ and $T$ have bases of the form
$\{g,ga,\ldots ga^{\dim S-1}\}$ and $\{ g', g'a,\ldots ,g'a^{\dim T
  -1}\}$ for some $g,g',a\in L$.
\end{theorem}

\begin{proof}
 From Section~\ref{sec:critical} it suffices to show that there exists
 a space $A$ of dimension $2$ such that $\dimn SA=\dimn S+1$. The
 result will therefore follow by induction if we can show that when
 $\dimn T>2$, there exists a subspace of $T'$ such that
 $\dimn T'=\dimn T-1$ and $\dimn ST'\leq\dimn S+\dimn T'-1$.
 Without loss of generality we may suppose that $L$ is the subfield
 generated by $S$ and $T$, whence its transcendence degree is finite,
 in other words $L$ is a function field. Since there always exists a
 place of $L$, by composition of places we may choose a place with
 values in an algebraic extension of $F$ which is $F$ because $F$ is
 algebraically closed. This place is equivalent to 
 a valuation ring with residual field equal to $F$ (e.g. Lang
 \cite[Ch. 7]{lang}). We may now choose a basis $(\tau_1,\tau_2,\ldots
 ,\tau_t)$ of $T$ such that $v(\tau_1)>v(\tau_2)>\cdots >v(\tau_t)$
  where $v$ is the valuation function. Setting $T'$ to be the subspace
  generated by $\tau_1,\ldots ,\tau_{t-1}$, we see that elements of
  $ST$ of minimum valuation cannot exist in $ST'$. Therefore
  $\dimn ST'<\dimn ST$ so that $\dimn ST'\leq \dimn S+\dimn T'-1$ and
  we are done.
\end{proof}

The above approach extends to every case when we can guarantee the
existence of a place of $L$ with values in $F$. However when $F$ is
not algebraically closed, it may happen that $L$ has no $F$-valued
places in which case we can only guarantee the existence of 
places with values in an algebraic extension of $F$. In this type of
situation we can hope to obtain a result through the valuation
approach only if we already have a version of Vosper's Theorem for
finite extensions. We illustrate this below by extending
Theorem~\ref{thm:main} to transcendental extensions of finite fields.

\begin{theorem}\label{thm:infinite}
  Let $F$ be the finite field $\F_q$  with $q$ elements and let $L/F$ be an
  infinite extension such that no element of $L\setminus F$ is
  algebraic over $F$. Let $S,T$ be subspaces of $L$ such that
  $\dimn S\geq 2,\dimn T\geq 2$ and
  $$\dimn ST=\dimn S+\dimn T-1.$$
  Then $S$ and $T$ have bases of the form
$\{g,ga,\ldots ga^{\dim S-1}\}$ and $\{ g', g'a,\ldots ,g'a^{\dim T
  -1}\}$ for some $g,g',a\in L$.
\end{theorem}

\begin{proof}
  As in the proof of Theorem~\ref{thm:closed}, we are done if we can
  exhibit a subspace $T'$ of codimension $1$ in $T$ such that $\dimn
  ST'\leq\dimn S+\dimn T'-1$.

  By the Lang-Weil estimation of the number of rational points of an
  algebraic variety, for any sufficiently large $m$ there exists a
  place of $L$ with values in the finite extension $\F_{q^m}$ of
  $F=\F_q$ \cite[Cor 4]{langweil}. From this we have a valuation 
  $v$ from $L$ into an ordered abelian group, together with a
  valuation ring $\OO$ in $L$ made up of all elements of non-negative
  valuation. If $P$ is the maximal ideal of $\OO$ made up of all elements of
  positive valuation, we have the isomorphism
  $\OO/P\xrightarrow{\sim}\F_{q^m}$. We chose $m$ to be prime and such
  that $m\geq \dimn ST+2$.
  
  Without loss of generality, translating $T$ if need be, we may
  suppose that the minimum valuation of elements of $T$ is $0$. 
  Let $T_P=T\cap P$ and decompose $T$ as $T=T_0\oplus T_P$ where $T_0$
  is any subspace of $T$ in direct sum with $T_P$. Note that since $0$
  is the minimum valuation in $T$, we have that $T_0\neq\{0\}$ and all
  non-zero elements of $T_0$ are of valuation $0$.
  Let $S=S_0+S_P$ be a similar decomposition.

  We clearly have $ST\supset S_{0}T_{0} + S_PT_P$.
  Now let $E$ be a maximal subspace of $S_{0}T_{0}$ all of whose
  non-zero elements have valuation zero. We therefore have
  $$ST\supset E \oplus S_PT_P.$$
  We make the remark that if $s\in S_{0}$ and if $\tau$ is an
  element of $T_P$ of minimum valuation, then the valuation of
  $s\tau$ equals the valuation of $\tau$ and does not belong to the
  set of valuations of the space $E+S_PT_P.$ Therefore
  $s\tau\not\in E+S_PT_P$ and
  \begin{equation}
    \label{eq:+1}
    \dimn ST \geq \dimn E + \dimn S_PT_P +1.
  \end{equation}
  
  The map $\pi : \OO\rightarrow\OO/P$ is injective on
  $E$ and we have
  $$\dimn E=\dimn\pi(E)=\dimn\pi(S_{0})\pi(T_{0}).$$
  By the Cauchy-Davenport inequality (Theorem~\ref{thm:Cauchy})
  applied in the extension $\OO/P$ of $F$, i.e. in $\F_{q^m}/\F_q$, we
  have that $\dimn E\geq \dimn S_{0} + \dimn T_{0} -1$. From 
  the Cauchy-Davenport inequality applied in $L$ we have that
  $\dimn S_PT_P\geq \dimn S_P + \dimn T_P -1.$
Since $\dimn ST =\dimn S+\dimn T-1$ by the hypothesis of the Theorem,
\eqref{eq:+1} implies that both the above inequalities must be equalities
so that we have 
  $$\dimn\pi(S_{0}T_{0}) = \dimn E= \dimn S_{0} +
  \dimn T_{0} -1.$$
Now if $\dimn S_{0}\geq 2$ and $\dimn T_{0}\geq 2$, because we
have chosen the extension $\F_{q^m}$ sufficiently large,
Theorem~\ref{thm:main} applies in $\F_{q^m}/\F_q$ and there exists a
subspace $T'_{0}$ of $T_{0}$ of codimension~$1$ such that
$\dimn \pi(S_{0})\pi(T'_{0}) < \dimn \pi(S_{0})\pi(T_{0})$
and $T'=T'_{0}\oplus T_P$ is the required subspace of $T$ of
codimension $1$. If either $\dimn S_{0}=1$ or  $\dimn T_{0}=1$,
then $T'=T'_{0} + T_P$ for any subspace $T'_{0}$ of
codimension $1$ of $T_{0}$ again yields the required subspace of $T$.
\end{proof}

\section{Concluding Comments}\label{sec:comments}
\begin{itemize}
\item
An alternative proof of Theorem~\ref{thm:closed} follows from a
version of Theorem~\ref{thm:sidon} for algebraically closed fields
$F$. To obtain this, we need the equivalent of Theorem~\ref{Theorem
  main}. We actually have a stricter upper bound. We only deal with
odd characteristic:
\begin{theorem} If $F$ is algebraically closed of odd characteristic, 
  and if $\CC\subset \QQ_n$ is a linear space,
  such that $\wt(\CC)\geq 3$, then $\dim(\CC)\leq n(n+1)/2-(4n-6)$.
\end{theorem}
\begin{proof}
If $F$ is algebraically closed, up to $F^*\times \Gl(n,F)$-equivalence, $x_1^2$ is the only
anisotropic form, from which it follows fairly easily that for a
quadratic form $Q$,
\begin{equation*}
\wt(Q)\geq 3 \Longleftrightarrow \rk(Q)\geq 5.
\end{equation*}
We identify $\QQ_n$ with the space $\SSS_n$ of symmetric matrices over
$F$. We introduce $\RR_4:=\{S\in \SSS_n \ : \ \rk(S)\leq 4\}$, which is
an algebraic variety over $F$ (because the rank condition is equivalent to the
condition that the minors of order $5$ vanish, and the minors are polynomials in
the matrix coefficients). The
assumption $\wt(\CC)\geq 3$ translates to $\CC\cap \RR_4=\{0\}$ which
implies
that $\dim(\CC)+\dim(\RR_4)\leq \dim(\SSS_n)= n(n+1)/2$ (see
\cite{Har}) so it remains to compute $\dim(\RR_4)$. The variety
$\RR_4$ splits into  $5$ orbits under the action of $\Gl(n,F)$, with
representatives of the form 
\begin{equation*}
\left(\begin{array}{c|c} 
S_0 & 0 \\
\hline
 0 &0
\end{array}
\right)
\end{equation*}
where 
\begin{equation*}
S_0\in \left\{ 
(0), (1), 
\left(\begin{array}{cc}
0&1\\1&0
\end{array}
\right),
\left(\begin{array}{ccc}
0&1&0\\1&0&0 \\ 0&0&1
\end{array}
\right),
\left(\begin{array}{cccc}
0&1&0&0\\1&0&0&0 \\ 0&0&0&1\\0&0&1&0
\end{array}
\right)
\right\}.
\end{equation*}

Each of these orbits can be identified with the quotient space
$\Gl(n,F)/O((\begin{smallmatrix} S_0 &0\\0&0\end{smallmatrix}))$ where 
$O((\begin{smallmatrix} S_0 &0\\0&0\end{smallmatrix}))$ denotes the
orthogonal group of the form $S:=(\begin{smallmatrix} S_0
  &0\\0&0\end{smallmatrix})$. We assume $S_0\neq (0)$ since in this
case the orbit is reduced to one point.
Let $U\in O(S)$, and let us write
$U=(\begin{smallmatrix} U_{1,1}
  &U_{1,2}\\U_{2,1}&U_{2,2}\end{smallmatrix})$
according to the block structure of $S$. Then, a straightforward
verification shows that the condition $U^t
SU=S$ translates to: $U_2=0$ and $U_1\in O(S_0)$. There are no
constraints on $U_3$; the only condition on $U_4$ is $U_4\in
\Gl(n-n_0,F)$ where $n_0=\rk(S_0)$ (so that $U\in \Gl(n,F)$). Putting
everything together, we find that 
$\dim(O(S))=\dim(O(S_0))+n_0(n-n_0)
+\dim(\Gl(n-n_0))=\dim(O(S_0))+n(n-n_0)$, and the corresponding orbit
in $\RR_4$ has dimension $nn_0-\dim(O(s_0))=nn_0-n_0(n_0-1)/2$. 
\end{proof}

\item Theorem~\ref{thm:infinite} follows also from
  Proposition~\ref{prop:sidon} and Theorem~\ref{thm:sidon}. However
  the alternative proof given in Section~\ref{sec:transcendental} serves to
  illustrate that versions of Vosper's theorem for an infinite
  dimensional extension will require a finite-dimensional theorem as a
  stepping stone.

\item What is the minimum dimension of the square of a Sidon space~? This
question is quite intriguing and very much open even in the case of
finite fields.
It is likely that the upper bound claimed in Theorem \ref{Theorem main} is far from tight,
and that it would be possible to greatly improve it by pursuing the
Delsarte methodology further. For finite fields this would yield a lower bound on the
minimum dimension of the square of a Sidon space that improves upon
Theorem~\ref{thm:sidon}. 

\item The central question left open is that of a complete characterisation of
  the critical spaces achieving equality in the Cauchy-Davenport
  inequality of Theorem~\ref{thm:Cauchy}. The example at the end of
  section~\ref{sec:sidon} shows that it is not true for all base
  fields that all such
  pairs $\{S,T\}$ of spaces have bases in geometric progression
  whenever $\dim(S),\dim(T), \codim(ST) \geq 2$.
\end{itemize}

\paragraph{{\bf Acknowledgements.}}
We are indebted to Vincent Beck, Jean-Marc Couveignes, Alain Couvreur and Qing Liu
for valuable comments.

\section{Appendix: the computation of the needed values of $\frac{\chi_s(t)}{|\OO_s|}$}

In this appendix, we compute explicit expressions for $\chi_s(t)/|\OO_s|$ when $t=(1,1)$,
$(2,0)$ and $(2,2)$. We start by recalling the well known  formulas for the
number of zeroes of  a quadratic form over a finite field, in terms of its type
(see \cite[Theorem 11.5]{Tay}).

\begin{lemma}\label{prop NQ}
 For $Q\in \QQ_n$, let
\begin{equation*}
Z_Q=|\{x\in \F_q^n \ : \ Q(x)=0\}|.
\end{equation*}
The value of $Z_Q$ only depends on the dimension $n$ and the type $t=(r,e)$ of $Q$, and will be
denoted $Z_{n,t}$. We have:
\begin{equation*}
Z_{n,t}=q^{n-1}+(1-e)q^{n-r/2-1}(q-1).
\end{equation*}
\end{lemma}

The wanted results are summarized in the following proposition. 
 It turns out that the formulas for $\chi_s(t)/|\OO_s|$ are slightly different depending on the
parity of $q$.

\begin{proposition}\label{prop Pst} We have, for $s=(r,e)$,
\begin{enumerate}
\item \underline{Case $t=(1,1)$}
\begin{enumerate}
\item If $p=2$ 
\begin{equation*}
\displaystyle \frac{\chi_s(t)}{|\OO_s|}=\left\{
\begin{array}{cl}
\displaystyle \frac{-1}{q^n-1} & \text{ if } e=1,2\\
1  & \text{ if } e=0
\end{array}
\right.
\end{equation*}

\item If $p\neq 2$ 
\begin{equation*}
\displaystyle \frac{\chi_s(t)}{|\OO_s|}=\frac{(1-e)q^{n-r/2}-1}{q^n-1}  
\end{equation*}

\end{enumerate}

\item \underline{Case $t=(2,0)$}
\begin{enumerate}
\item If $p=2$ 
\begin{equation*}
\displaystyle \frac{\chi_s(t)}{|\OO_s|}=\left\{
\begin{array}{cl}
\displaystyle \frac{q^{2n-r-1}-2q^{n-1}+1}{(q^n-1)(q^{n-1}-1)} & \text{ if } e=1,2\\
\displaystyle \frac{q^{2n-r-1}-q^{n}-q^{n-1}+1}{(q^n-1)(q^{n-1}-1)} & \text{ if } e=0
\end{array}
\right.
\end{equation*}

\item If $p\neq 2$ 
\begin{equation*}
\frac{\chi_s(t)}{|\OO_s|}= \frac{q^{2n-r-1}-2q^{n-1}+1 -(1-e)q^{n-r/2-1}(q-1)}{(q^n-1)(q^{n-1}-1)}
\end{equation*}
\end{enumerate}

\item \underline{Case $t=(2,2)$}
\begin{enumerate}
\item If $p=2$ 
\begin{equation*}
\displaystyle \frac{\chi_s(t)}{|\OO_s|}=\left\{
\begin{array}{cl}
\displaystyle \frac{-q^{2n-r-1}+1}{(q^n-1)(q^{n-1}-1)} & \text{ if } e=1\\
\displaystyle \frac{q^{2n-r-1}-q^n-q^{n-1}+1}{(q^n-1)(q^{n-1}-1)}  & \text{ if } e=0\\
\displaystyle \frac{q^{2n-r-1}+1}{(q^n-1)(q^{n-1}-1)}  & \text{ if }e=2
\end{array}
\right.
\end{equation*}

\item If $p\neq 2$ 
\begin{equation*}
\displaystyle \frac{\chi_s(t)}{|\OO_s|}=\left\{
\begin{array}{cl}
\displaystyle \frac{-q^{2n-r-1}+1}{(q^n-1)(q^{n-1}-1)} & \text{ if } e=1\\
\displaystyle \frac{q^{2n-r-1}+1-(1-e)q^{n-r/2-1}(q+1)}{(q^n-1)(q^{n-1}-1)}  & \text{ if } e=0,\,2
\end{array}
\right.
\end{equation*}

\end{enumerate}
\end{enumerate}
\end{proposition}

\begin{proof}
In the following, $Q_t=\sum_{i\leq j} a_{i,j}^t x_i
x_j$ denotes the representative of quadratic forms of type $t$ given
in Proposition \ref{proposition equivalence Q}, and analogously $B_s=(
b_{i,j}^{s})$
is the representative of symmetric bilinear forms of
type $s$ from Proposition
\ref{proposition equivalence S}. The order of the stabilizer subgroup
$G_B$ of $B$ in $G$  depends only on the type $s$ of $B$; it will be
denoted $|G_s|$.
We introduce the following notation: if $Q=\sum_{i\leq j} a_{i,j}
x_ix_j$ and $B=(b_{i,j})$, we let $b(Q,B):=\sum_{i\leq j} a_{i,j}
b_{i,j}$. We recall that
$(Q,B)=\alpha(b(Q,B))$, where $\alpha:(\F_q,+)\to (\C^*,\times)$ is a fixed
non trivial character. Then, 
\begin{align*}
\chi_s(t)&= \sum_{B\in \OO_s} (Q_t,B) =\frac{1}{|G_s|} \sum_{g\in
  G} (Q_t,B_s^g)\\
&=\frac{1}{|G_s|} \sum_{u\in  \Gl(n,q)} \big(\sum_{a\in {\F_q^*}} (Q_t,B_s^{(a,u)})\big).
\end{align*}
Due to the standard properties of characters,
\begin{equation*}
\sum_{a\in {\F_q^*}} (Q_t, B_s^{(a,u)})=\left\{
\begin{array}{ll}
-1 & b(Q_t,B_s^u)\neq 0\\
q-1 & b(Q_t,B_s^u)=0
\end{array}
\right.
\end{equation*}
We introduce:
\begin{equation*}
N_{s,t}:=|\{u\in \Gl(n,q) :\ b(Q_t,B_s^u)=0\}|
\end{equation*}
and we have:
\begin{equation}\label{eq Pst}
\chi_{s}(t)=\frac{1}{|G_s|}(qN_{s,t}-|\Gl(n,q)|)=\frac{|\OO_s|}{q-1}\Big(\frac{qN_{s,t}}{|\Gl(n,q)|}
-1\Big).
\end{equation}
 Let $\{v_1,\dots,v_n\}$ denote the column vectors of the matrix $u\in
  \Gl(n,q)$. We have
\begin{equation}\label{equa 1}
b(Q_t,B_s^u)=\sum_{i\leq j} a_{i,j}^t (v_i^t B_s v_j).
\end{equation}
We introduce the quadratic form
$Q(x):=x^tB_sx$ and will denote its type by $t'$; it can be
readily verified that,  if $p\neq 2$, $t'=s$, and if
$p=2$,  $t'=(1,1)$ for $e=1,2$ and $t'=0$ for
$e=0$. 

\noindent\underline{Case 1: $t=(1,1)$.} Then, $Q_t=x_1^2$ and 
$b(Q_t,B_s^u)=v_1^tB_s v_1$. So, with the notation above
and that of
Proposition \ref{prop NQ}, 
\begin{align*}
N_{s,t}&=(Z_{n,t'}-1)(q^n-q)\dots
(q^n-q^{n-1})\\
&=(Z_{n,t'}-1)\frac{|\Gl(n,q)|}{q^n-1}.
\end{align*}
The expressions for $\chi_s(t)/|\OO_s|$ follow then immediately from the formulas
for $Z_{n,t'}$ in Proposition \ref{prop NQ} and from \eqref{eq Pst}.

\smallskip

\noindent\underline{Case 2: $t=(2,0)$.} In this case, $Q_t=x_1x_2$ and 
$b(Q_t,B_s^u)=v_1^t B_s v_2$.  For $w=(v_1,v_2)\in \F_q^{2n}$, 
$Q(w):=v_1^t B_s v_2$ defines a quadratic form which can be
easily seen to be of type $(2r,0)$. So, the number of $w=(v_1,v_2)$ such
that $Q(w)=0$ is equal to $Z_{2n,(2r,0)}$. Because $v_1,v_2$ are the two first column
vectors of $u\in \Gl(n,q)$, they should not be linearly dependent. The
number of $(v_1,v_2)$ such that $Q(w)=0$ and $\{v_1,v_2\}$ are linearly
dependent
is: $2q^n-1+(q-1)(Z_{n,t'}-1)$ so we obtain:
\begin{equation*}
N_{s,t}=(Z_{2n,(2r,0)}-  (q-1)(Z_{n,t'}-1) -2q^n+1)\frac{|\Gl(n,q)|}{(q^n-1)(q^n-q)}.
\end{equation*}

\smallskip

\noindent\underline{Case 3: $t=(2,2)$.} 
In this case, $Q_t$ depends on the parity of $p$.
\smallskip
If $p=2$, $Q_t=x_1^2+x_1x_2+bx_2^2$ and $b(Q_t,B_s^u)=v_1^t B_s v_1+ v_1^t B_s v_2+b v_2^t B_s v_2$.
 For $w=(v_1,v_2)\in \F_q^{2n}$, let $Q(w):=v_1^t B_s v_1+ v_1^t
 B_s v_2+b v_2^t B_s v_2$. 
Again, we count the
number of $(v_1,v_2)$ such that $Q(w)=0$ and $\{v_1,v_2\}$ are linearly
dependent, and find $1+(q+1)(Z_{n,t'}-1)$. We need to determine the
type $t''$ of $Q$. With some work, we find:
\begin{align*}
&\text{If } s=(r,0),\, t''=(2r,0)\\
&\text{If } s=(r,1),\, t''=(2r,2)\\
&\text{If } s=(r,2),\, t''=(2r,0)
\end{align*}
so 
\begin{equation*}
N_{s,t}=(Z_{2n,t''}-  (q+1)(Z_{n,t'}-1) -1)\frac{|\Gl(n,q)|}{(q^n-1)(q^n-q)}.
\end{equation*}

\smallskip
If $p\neq 2$, $Q_t=x_1^2-bx_2^2$ and $b(Q_t,B_s^u)=v_1^t  B_s v_1 -b v_2^t B_s v_2$.
 For $w=(v_1,v_2)\in \F_q^{2n}$, let $Q(w):=v_1^t B_s v_1-b v_2^t B_s v_2$. 
The number of $(v_1,v_2)$ such that $Q(w)=0$ and $\{v_1,v_2\}$ are linearly
dependent is as in the previous case of $p=2$ equal to
$1+(q+1)(Z_{n,t'}-1)$. We find the same type $t''$ for $Q$ so again
\begin{equation*}
N_{s,t}=(Z_{2n,t''}-  (q+1)(Z_{n,t'}-1) -1)\frac{|\Gl(n,q)|}{(q^n-1)(q^n-q)}.
\end{equation*}
\end{proof}

\end{document}